\newcommand{\MA}{\mathfrak{A}}

\newcommand{\BI}{\mathbb{I}}
\newcommand{\CX}{\mathcal{X}}

\newcommand{\BX}{\mathbb{X}}
\newcommand{\ZZ}{\mathbb{Z}}
\newcommand{\CW}{\mathcal{W}}

\newcommand{\Div}{\mbox{div}}
\newcommand{\DeltaA}{\Delta}

\newcommand{\Mr}{{{ \mathfrak R }}}

\newcommand{\Ltwo}{L^2}

\newcounter{lil11}

\newcounter{lil22}
\newenvironment{steps2}
{\em \begin{list} {Step (\roman{lil22}):} {\usecounter{lil22} \em
\setlength{\leftmargin}{0.3cm}
\setlength{\topsep}{0.5cm} \setlength{\itemsep}{0.0cm}
\setlength{\parsep}{0.1cm} \setlength{\itemindent}{0.4cm}
\setlength{\parskip}{0.0cm}}} {\end{list}}

\newcommand{\baray}{\begin{array}{rcl}}
\newcommand{\earay}{\end{array}}
\newcommand{\barray}{\begin{array}{rcl}}
\newcommand{\earray}{\end{array}}

\newcommand\dela[1]{}

\newcommand{\bcase}{\begin{cases}}
\newcommand{\ecase}{\end{cases}}

\newcommand\del[1]{}

\del{

\newcommand\del[1]{}
}

\def\eps{\varepsilon}

\newcommand{\lk}{\left}
\newcommand{\lqq}{\lefteqn}
\newcommand{\rk}{\right}
\newcommand{\la}{{\langle}}
\newcommand{\ra}{{\rangle}}

\newcommand{\LL}{{\rm I \kern -0.2em L}}

\newcommand{\ep} {\varepsilon }

\newcommand{\be} {\begin{enumerate} }
\newcommand{\ee} {\end{enumerate} }

\newcommand{\CO}{{{ \mathcal O }}}

\newcommand{\CT}{{{ \mathcal T }}}

\newcommand{\CH}{{{ \mathcal H }}}
\newcommand{\CS}{{{ \mathcal S }}}

\newcommand{\CM}{{{ \mathcal M }}}

\newcommand{\BF}{{{ \mathbb{F} }}}
\newcommand{\CF}{{{ \mathcal F }}}

\newcommand{\CL}{{{ \mathcal L }}}

\newcommand{\RR}{{\mathbb{R}}}

\newcommand{\NN}{\mathbb{N}} 

\newcommand{\SubX}{\mathcal{X}_\MA(R_1,R_2)}

\newcommand{\di}{\mbox{div}}

\newcommand{\PP}{{\mathbb{P}}}

\newcommand{\EE}{ \mathbb{E} }

\newcommand{\DEQS}{\begin{eqnarray*}}
\newcommand{\EEQS}{\end{eqnarray*}}
\newcommand{\DEQSZ}{\begin{eqnarray}}
\newcommand{\EEQSZ}{\end{eqnarray}}
\newcommand{\DEQ}{\begin{eqnarray}}
\newcommand{\EEQ}{\end{eqnarray}}

\newcounter{lil1q}
\newenvironment{steps}
{\begin{list} { \bf Step (\Roman{lil1q})}
{ \usecounter{lil1q}
\setlength{\leftmargin}{0.0cm}
\setlength{\topsep}{0.2cm}
\setlength{\itemsep}{0.0cm}
\setlength{\parsep}{0.1cm}
\setlength{\itemindent}{0.8cm}
\setlength{\parskip}{0.0cm}}}
{\end{list}}
\documentclass[11pt,english]{amsart}
\usepackage{color}
\usepackage{amsmath}   
\usepackage{amsfonts,amssymb, color}
\usepackage{latexsym}
\input colordvi

\usepackage{enumitem}

\usepackage{amsfonts}

\usepackage{enumitem}

\theoremstyle{plain}

\newtheorem{theorem}{Theorem}[section]
\newtheorem{notation}{Notation}[section]
\newtheorem{claim}{Claim}[section]
\newtheorem{corollary}[theorem]{Corollary}
\newtheorem{hypo}[theorem]{Assumption}
\newtheorem{example}[theorem]{Example}

\newtheorem{definition}[theorem]{Definition}
\newtheorem{remark}[theorem]{Remark}

\newtheorem{proposition}[theorem]{Proposition}
\newtheorem{tproposition}[theorem]{Technical Proposition}

\numberwithin{equation}{section}

\numberwithin{equation}{section} \allowdisplaybreaks

\begin{document}

\title[The 2d stochastic Keller--Segel model in porous media]{
Martingale solution to a stochastic chemotaxis system with porous medium diffusion}

\author{Erika Hausenblas}
   \address{%
   Department of Mathematics,
	Montanuniversitaet Leoben,
	Austria.}
\email{erika.hausenblas@unileoben.ac.at}

\author[Debopriya Mukherjee]{Debopriya Mukherjee}

\address{%
   Department of Mathematics,
	Montanuniversitaet Leoben,
	Austria.}
\email{debopriya.mukherjee@unileoben.ac.at}

\author[Ali Zakaria]{Ali Zakaria}
\address{%
Department of Mathematical Sciences, University of South Africa, Florida, 0003}
\email{zakaria@aims.ac.za}

\date{\today}
\thanks{The first author of the paper is supported by Austrian Science Foundation, project number P 32295. The second author is supported by Marie Sk{\l}odowska-Curie Individual Fellowships H2020-MSCA-IF-2020, 888255.
}

\begin{abstract}
In this paper, we study the classical Keller–Segel system on a two-dimensional domain perturbed by a pair of Wiener processes, where the leading diffusion term is replaced by a porous media term.
In particular, we investigate the coupled system
\begin{align*}
\lk\{
\barray \dot {u} & =& r_u\Delta   |u|^{\gamma-1}u- \chi \Div( u\nabla v) +\sigma_u\, u \circ dW_1,
\\
\dot{v}& =&r_v \Delta v  -\alpha v+ \beta u +\sigma_v\, v \circ dW_2,
\earray \rk.
\end{align*}
for $\gamma>1$, with initial condition $(u_0,v_0)$ on a filtered probability space  $\mathfrak{A}$ and $(W_1,W_2)$ be a pair of time homogeneous spatial Wiener processes over $\mathfrak{A}$. Here $u$ is the cell density and $v$ is the concentration of the chemical signal, $\sigma_u$ and $\sigma_v$ are positive constants. The positive  terms $r_u$ and $r_v$  are the diffusivity of the  cells and chemoattractant, respectively.
The positive value  
 $\chi$ is the chemotactic sensitivity,
  $\alpha\ge0$ is the so-called damping constant and $\beta\ge0$ is the production weight corresponding to $u$. 
  \par
Since the randomness is intrinsic, interpretation of the stochastic integral in Stratonovich sense is natural. We construct a solution (integral) operator, and establish its continuity and compactness properties in an appropriately chosen Banach space. In this manner, we formulate a stochastic version of the Schauder--Tychonoff Type Fixed Point Theorem which is specific to our problem to obtain a solution. In kind, we achieve the existence of a martingale solution. 
%
\end{abstract}

\maketitle




\section{Introduction}
The importance of model organisms with the goal of the uniformity of the mathematical structure and its related biological factors is fundamental. The celebrated Keller--Segel model introduced by Keller and Segel \cite{KSS}, and Patlak \cite{patlak} illustrate the aggregation of Dictyostelium discoideum in view of pattern formation. For adequate references on the phenomenological analysis of the aforesaid class of models we refer to Horstmann \cite{horstmann1,horstmann2}, Hillen and Painter \cite{hillen1},  Bellomo {\em et al.} \cite{bellomo1}, the works of Biler \cite{biler2}, and  of Perthame \cite{perthame}.

Chemotaxis can be defined as the movement or orientation of a population (bacteria, cell or other single or multicellular organisms) induced by a chemical concentration gradient either towards or away from the chemical signals. Chemotaxis is a fundamental and universal phenomenon  which includes (but not limited to)
\begin{itemize}
\item fertilization and reproduction; see \cite{Eisenbach,R.Lord_ferti} etc.
\item bacterial motility; see \cite{R.Dillon_bacteria,Chun_1988} etc.
\item development of axons in nervous system; see Chapter 8 in \cite{Eisenbach} and references therein
\item molecular mechanisms in amobae; see \cite{Michael_amobae} etc.
\end{itemize}

In the classical Keller–Segel model for chemotaxis, the chemoattractant
is emitted by the cells that react according to biased random walk inducing linear
diffusion operators. In view of \cite{Szymanska_2009}, migration of the cells in porous media and the cell motility is a nonlinear function of the cell density. Hence, there is a genuine need to study the chemotaxis systems
with porous medium diffusion. It is well-known that concentration gradients in
porous media display a complex topology and are highly variable in terms of magnitude and direction, leading to the occurrence of non-linear effects on bacteria motility.

%

\del{
However, flow fields in
porous media display a complex topology and are highly variable in terms of magnitude and direction, leading to the occurrence of non-linear effects on bacteria motility.
This complexity limited so far our understanding and modeling capability of chemotaxis in porous media. In order to properly model transport of microorganisms and
their response to the attractants, numerical schemes should accurately reproduce low
concentrations and concentration gradients. In particular, the correct computation of
concentration gradients which drive chemotaxis represents a challenge for numerical
schemes.
A number of numerical models for chemotaxis in porous media

These well
established numerical schemes suffer of artificial numerical diffusion, leading to significant errors in the reproduction of solute gradients, thereby hampering their ability
to correctly reproduce the movement of chemotactic bacteria. These features are very appealing for the study
of chemotaxis in porous media.

However, extending the standard SPH to chemotaxis
is not straightforward because bacteria are affected by two velocity fields, advection
and chemotaxis velocity, while the chemo-attractant is affected only by advection.
Furthermore, SPH does not reproduce correctly concentration gradients, particularly
when particles, representing elements of the carrier fluid, are non-uniformly distributed, as it occurs in heterogeneous velocity fields (Boso et al. 2013; Avesani et al.
2014, 2015). Furthermore, an interpolation scheme should be introduced to compute
the concentration of the chemoattractant at bacteria positions, which is an additional
source of numerical error.
Motivated by these considerations, the primary goal of this paper is to extend a new
class of
}

Finally, Porter {\em et al.} \cite{Porter2011} developed a multiscale model of chemotaxis in porous
media where transport of bacteria is expressed in terms of effective medium parameters.
The simplest form of the model is
%
\DEQSZ\label{sys1}
\qquad\lk\{ \barray
d{u}(t,x)&=& \Big( r_u \Delta u^{[\gamma]} (t,x) 
-\chi \mbox{div} \big( u(t,x)\nabla v(t,x)\big)\Big)\, dt, \quad  t \geq 0,\,\,x \in \CO,
\\
d{v}(t,x) &=& \big(r_v \Delta v(t,x)+\beta u(t,x) 
-\alpha v(t,x)\Big)\, dt,  \quad t \geq 0,\,\,x \in \CO, \\
(u(0,x),v(0,x))&=&(u_0(x),v_0(x)), \quad x \in \CO,
\phantom{\big|}
\earray\rk.
\EEQSZ
for $\gamma>1$ and $\CO$ is a bounded domain in $\RR^d$. Here,  $\xi^{[\gamma]}$ is an abbreviation for $\xi |\xi|^{\gamma-1}$ for $\xi \in \RR$, and  $\Delta$ denotes the Laplacian with Dirichlet (or Neumann) boundary condition.
Furthermore, $u$ denotes the cell density and $v$ is the concentration of the chemical signal. The term $-u \cdot \nabla v$ is the chemotactic flux, which describes the transportation of the amount of the bacteria in the direction of the flow. This gives rise to the evolution of bacterial concentration by diffusion and transport along the flow. This essentially produces $-\chi \mbox{div} \big( u\nabla v\big)$ illustrating
cross diffusive effects into the model, where the positive constant $\chi$ represents the chemotactic sensitivity. Here, $\Delta u^{[\gamma]}$ is the migration of the bacteria, for which, the motility depends on the bacterial density and $\Delta v$ is the diffusion of chemoattractant. The positive  terms $r_u$ and $r_v$  are the diffusivity of the  cells and chemoattractant, respectively. In the signal concentration model,  $\alpha\ge0$ is the so--called damping constant and $\beta\ge0$ is the production weight corresponding to $u$. The leading term is not parabolic at all points, but only
degenerate parabolic, a fact that has mathematical consequences, both qualitative
and quantitative.
For a quick survey, we again refer to 
\cite{ perthame, horstmann1, horstmann2, bellomo1, biler2, hillen1}. For a brief review on the behaviour of classical Keller–Segel model by a degenerate diffusion of porous medium type in the deterministic framework, we refer the articles of Carrillo {\em et al.} \cite{carrillo1,carrillo2,carrillo3} and the references cited therein.

The Keller--Segel model is a  macroscopic model derived from the limiting behaviour of the microscopic model; see, e.g.\ \cite{stevens1}. Here, one relies on fundamental balance laws and Fick's law of diffusion. Consequently, significant aspects of microscopic dynamics such as fluctuations of molecules are disregarded. Hence, in the derivation of the above macroscopic equations, fluctuations around the mean
value are neglected.
Secondly, in natural systems random disturbances and so-called  environmental noise is inevitable and may create together with the nonlinearity a change of the dynamical behaviour.
For a more realistic model, it is necessary to consider essential features of the natural environment which are non-reproducible. Hence, the model should include random spatio-temporal
forcing.

The randomness leads to a variate of new phenomena and may have highly non--trivial impact on the behaviour of the solution.
It should be stressed that  adding a stochastic driving term to a partial differential equation can have highly non--trivial impact implication on the behaviour of the solution.
The presence of the stochastic term (noise) in the models often leads to qualitatively new types of behaviour, which is most helpful in understanding  the real processes and is also often more realistic. 
For example, there exist deterministic systems of PDEs, the Navier-Stokes equations, for example, which have non-ergodic invariant measure. However, adding a noise term leads to the existence and uniqueness of an invariant measure, and, hence, an ergodic invariant measure.
Besides, researcher from applied science investigate in biological systems disturbed by some noise.
In Karig {\em et al.} \cite{karig}, the authors explore whether the stochastic extension leads to a broader range of parameter with Turing patterns by a genetically engineered synthetic bacterial population in which the signalling molecules form a stochastic activator{\textendash}inhibitor system.
Kolinichenko and Ryashko \cite{noise4}, respectively,  Bashkirtseva {\em et al.} \cite{noise5} addresses  multistability and noise-induced transitions between different states.

In this article, we perturb the density of cells $u$ and concentration of chemoattractant $v$ by
time-homogeneous spatial  Wiener processes.
Due to the non-linearities in the system, one fails to use semigroup approach for the equation perturbing cell density. Thus the standard methods to
show existence and uniqueness of solutions cannot be applied here. Here we formulate a stochastic version of Schauder-Tychonoff type Fixed Point Theorem which is specific to our problem to obtain a solution on $[0,T]$ for $T>0$. In this manner, we achieve only the existence of a martingale solution, but not the uniqueness of the solution.

Solvability and boundedness to the chemotaxis model with porous medium diffusion in determinstic set-up has gained much interest in recent years (see for instance \cite{Senba_2011, Chunhua_2020, Chunhua_2019}, and the references therein). The authors are not aware of any work which treats the stochastic modelling of the coupled Keller--Segel model with porous medium diffusion. Recently, the authors in \cite{EH+DM+TT_2020,EH+DM+JL_2020} have treated the simple stochastic Keller--Segel model in one and two dimensions. In these works, the authors rely on the entities controlled by a Lyapunov functional. We refer the very recent work \cite{tusheng}, where the deterministic Keller--Segel model is coupled with the stochastic Navier–Stokes equations.
The stochastic porous media equation itself is the topic of the recent book of Barbu, Da Prato and R\"ockner, \cite{roeckner}, where we used several results for the stochastic porous media equation.  In the article of Dareiotis, Gerencser, and Gess \cite{gess2020a}
the solvability of the stochastic porous media equation in dimension one with space time white noise is shown.
In the article of Dareiotis, Gess, and Tsatsoulis \cite{gess2019a} the long time behaviour of the stochastic porous media equation is treated.

\begin{notation}
For a Banach space $E$ and $0\le a<b<\infty$, let $C^{\zeta}_b([a,b];E)$ denote a set of all continuous and bounded functions $u:[a,b]\to E$ such that
$$
\| u\|_{C_b^{\zeta}([a,b];E)} :=\sup_{a\le t\le b} |u(t)|_E +\sup_{a\le s,t\le b\atop t\not= s} \frac {|u(t)-u(s)|_E}{|t-s|^\zeta}
$$
is finite. The space $C_b^{\zeta}([a,b];E)$ endowed with the norm $\| \cdot\|_{C_b^{\zeta}([a,b];E)}$ is a Banach space.
\end{notation}

\begin{notation}
Let $1\le p < \infty$ and $s\in\RR$,  then $H^s_p$ denotes the Bessel Potential space or Sobolev space of fractional order defined by
\DEQS
H_p^s(\RR^d ) :=\lk\{ f\in \mathcal{S}'(\RR^d): |f|_{H^s_p} := \big| \big( (1+\xi^2 )^{\frac s2}  \hat{ f} \,
\big)^\vee \big|_{L^p}
<\infty\rk\}.
\EEQS
Here, we denote the Fourier transform of a function $f$ by $\hat f$ and its inverse by $f^\vee$.
Presently, $H_p^s(\mathcal O)$ is the restriction of $H_p^s(\RR^d)$ to $\mathcal{O}$, with
\[
\|f\|_{H_p^s(\mathcal O)} 
=\inf_{g|_{\mathcal{O}}=f,\\
\,\, g \in H_p^s(\RR^d)} \|g\|_{H_p^s(\RR^d)}.
\]
Here, $g|_{\mathcal{O}} \in \mathcal{D}'(\mathcal{O})$ denotes the restriction of $g \in \mathcal{D}'(\RR^d)$ to $\mathcal{O}$ in the sense of the theory of distributions.
\end{notation}

\section{Problem description and main result}
In this section we introduce the definition of martingale solution to the stochastic system and present our main result. Before, we formulate the necessary assumptions on the noise and the initial conditions.

\medskip

Let $\mathfrak{A}=(\Omega, \CF,\BF,\PP)$ be a complete probability space
and filtration $\BF=(\CF_t)_{t\ge 0}$  satisfying the usual conditions i.e.,
\begin{itemize}
\item[(i)]
 $\mathbb{P}$ is complete on $(\Omega, \CF)$,
 \item[(ii)]
for each $t\geq 0$, $\CF_t$ contains all $(\CF,\mathbb{P})$-null sets,
\item[(iii)]
  and the filtration $\BF$ is right-continuous.
\end{itemize}
Let $\CO\subset \RR^2$ be a bounded domain, with smooth boundary (or the rectangle $\CO=[0,1]^2$).
\del{Let us define the Laplacian with the Neumann boundary conditions by
\begin{equation}
\label{eqn:4.3} \left\{
\begin{array}{lcl}
D(A) &:= &\{ u \in H^2(\CO):\frac {\partial u}{\partial n}(x)=0\;\mbox{on} \;\partial \CO \},\cr
A u&:=&\Delta u= \sum_{i=1}^ d {\partial^2_{x_i}}u, \quad u \in D(A),
\end{array}
\right.
\end{equation}}
Let $H_1$ and $H_2$ be two Hilbert spaces,
 and $W_j$, $j=1,2$, be two cylindrical Wiener processes defined on $H_1$ and $H_2$,  respectively.
In this paper, we consider the following system of equations
\DEQSZ\label{sys1noise}
\lk\{ \barray
d{u}(t)&=& \Big( r_u \DeltaA  u^{[\gamma]} (t) 
-\chi \mbox{div} \big( u(t)\nabla v(t)\big)\Big)\, dt+\sigma_u\, u(t)\circ dW_1(t)
\\
d{v}(t) &=& \big(r_v \DeltaA  v(t)+\beta u(t) 
-\alpha v(t)\Big)\, dt +\sigma_v\, v(t)\circ dW_2(t),\,\quad t\in \RR_0^+,\phantom{\big|}
\earray\rk.
\EEQSZ
where $\sigma_u$ and $\sigma_v$ are positive constants
and  the positive terms $r_u$ and $r_v$  are the diffusivity of the cells and chemoattractant, respectively. Also, $\alpha\ge0$ is the so--called damping constant and $\beta\ge0$ is the production weight corresponding to $u$.
Here $A=\Delta$ denotes the Laplace operator with Neuman boundary conditions,  i.e.\ $D(A):= \{ u \in H^2(\CO):\frac {\partial u}{\partial n}(x)=0$ on $\partial \CO \}$,
where $n$ denotes the (typically exterior) normal to the boundary $\partial \CO$.
Since the randomness is intrinsic, interpretation of the stochastic integral in Stratonovich sense is natural. For a detailed explanation of the Stratonovich integral, we refer to the book by Duan and Wang \cite{duan} or to the original work of Stratonovich \cite{stratonovich}.
%
%
To show the existence of the solution, the Wiener perturbation have to satisfy regularity assumptions. Both processes $W_1$ and $W_2$ are cylindrical Wiener processes over real-valued Hilbert spaces $H_1$ and $H_2$, respectively.  Then, due to the spectral representation of Wiener processes, $W_1$ and $W_2$ can be written formally
by the sum (possibly infinite)
\[
 W_1(t,x)=\sum_{k\in\BI_1}
\psi^{(1)} _k(x)\beta^{(1)}_k(t) \quad\text{and}\quad
 W_2(t,x)=\sum_{k\in\BI_2}
 \psi^{(2)} _k(x)\beta^{(2)}_k(t),
\]
where $\{\psi^{(1)}_k:k\in\mathbb{I}_1\}$ and $\{\psi^{(2)}_k:k\in\mathbb{I}_2\}$ are  some orthonormal basis in $H_1$ and $H_2$, $\mathbb{I}_1$ and $\mathbb{I}_2$ are the corresponding index sets,  $\{\beta_k^{(1)}:k\in\mathbb{I}_1\}$ and   $\{\beta_k^{(2)}:k\in\mathbb{I}_2\}$ are two mutually independent standard Brownian motions over $\MA$.
In order to get the existence of a solution, we introduce the following hypothesis.
\begin{hypo}\label{wiener}
Let us assume that  the embeddings $H_1 \hookrightarrow L^\infty(\CO)$ and  $H_2 \hookrightarrow L^3(\CO)$ are $\gamma$--radonifying.
 \end{hypo}
%
%
\begin{hypo}\label{wiener1} Let $H_1$ and $H_2$ be isomorphic to Bessel potential spaces. For example,
$H_1=H^{\delta_1}(\CO)$ for $\delta_1>1$, and
$H_2=H^{\delta_2}(\CO)$ for $\delta_2>1$.   
Let us assume that there exist $\delta_i>1;\,\,i=1,2$ such that  the embeddings $ H^{\delta_i}_2(\CO)$ in $H^1_2$
 is a Hilbert-Schmidt for $i=1,2$.
 \end{hypo}

\begin{remark}
Let
$\{\phi_k:k\in\NN\}$ be the eigenfunctions of $(A,D(A))$ with the corresponding eigenvalues $\{\nu_k:k\in\NN\}$. Then, $\{\phi_{\delta_i,k}:k\in\NN\}$ with $\phi_{\delta_i,k}:=\nu_k ^{-\delta_i}\phi_k$ is an orthonormal basis in $H^{\delta_i}_2(\CO)$ for $i=1,2$.
Then, we
can write for  $W_1$ and $W_2$ as the following sum 
\[
 W_1(t,x)=\sum_{k\in\mathbb{I}}
\psi^{(\delta_1)} _k(x)\beta^{(1)}_k(t) \quad\text{and}\quad
 W_2(t,x)=\sum_{k\in\mathbb{I}}
 \psi^{(\delta_2)} _k(x)\beta^{(2)}_k(t).
\]
%
\end{remark}
\begin{example}
 In the case of a single dimension, a complete orthonormal system of the underlying Lebesgue space~$\Ltwo([0,2\pi])$ is given by sine and cosine functions
\begin{equation}
\theta_m(x)=
\begin{cases}
{\sqrt{2}} \, \sin\big({m} x\big) &\!\!\text{if } m \geq 1\,, \\
{\sqrt{2}} &\!\!\text{if } m = 0\,, \\
{\sqrt{2}}\, \cos\big({m} x\big) &\!\!\text{if } m \leq - 1\,.
\end{cases}
\end{equation}
The extension to $[0,1]^2$ relies on tensor products, i.e., for a multiindex $m= (m_1,m_2)\in \ZZ^2$ we  have
\begin{equation}
\Theta_m(x)=  \theta_{m_1}(x_1) \,\theta_{m_2}(x_2)\,, \, \quad x=(x_1,x_2)\in \CO.
\end{equation}
The corresponding eigenvalues are given by
\begin{equation}
\label{eq:lambda}
\lambda_m = - \, 4{\pi^2} ( m_1^2+m_2^2) \,, \quad m = (m_1,  m_2) \in \ZZ^2\,.
\end{equation}\end{example}

\begin{example}
Let $\CO=\mathcal{D}:=\{x\in\RR^d:|x|\le 1\}$ be the unit disk, then for a multiindex $m=(m_1,m_2)\in\NN\times \mathbb{Z}$, we have as eigenfunction
 $\psi(x_1,x_2)= J'_{m_1} (\kappa_{m_1,m_2}r)\theta_{m_2}(\tilde \theta)$, where $r=\sqrt{x_1^2+x_2^2}$, $\tilde \theta=\tan(x_1/x_2)$,  $J_{m_1}$ is the $m_1^{th}$ Bessel function, $\kappa_{m_1,m_2}$ is the $m_2^{th}$ root of $J'_{m_1}$.
The eigenvalues  $\{\lambda_m:m\in\NN\times \mathbb{Z}\}$ are given by $\lambda_m=\kappa^2_{m_1,m_2}$, $m\in  \NN\times \mathbb{Z}$. Note, the $n^{th}$ maxima of the $j^{th}$ Bessel function is for $n,j\to \infty$ approximately at $n$, hence $\lambda_m\sim m_2^2$.
\end{example}

%

Since the cell density,  $u$, and the concentration of the chemical signal, $v$, of the chemotaxis system have to be non-negative, the initial conditions $u_0$ and $v_0$ have to be non-negative  as well. Besides, we have to impose some regularity assumptions to get the existence of a solution.  
\begin{hypo}\label{init}
Let $u_0 \in H^{-1}_2(\CO)$ and $v_0\in L^ 2 (\CO)$ be two random variables over $\mathfrak{A}$ such that
\begin{enumerate}
  \item $u_0\ge 0$ and $v_0\ge 0$;
  \item $(u_0,v_0)$ is $\CF_0$--measurable;
   \item $\EE|u_0|_{L^{\gamma+1}}^{\gamma+1}$ and  $\EE|\nabla v_0|_{L^{4}}^{4}<\infty$. 
\end{enumerate}
\end{hypo}
%
%
%
As mentioned before, in the proof for the existence of the solution, we are using compactness arguments, which causes the loss of the original probability space. This essentially means the solution will only be a weak solution in the probabilistic sense. Thereupon, we have to construct another probability space and  obtain a martingale solution.
\begin{definition}\label{Def:mart-sol}
A {\sl  martingale solution}   to the problem
\eqref{sys1noise} is a system
\begin{equation}
\left(\Omega ,{{\mathcal{F}}},{\mathbb{F}},\mathbb{P},
(W_1,W_2), (u,v)\right)
\label{mart-system}
\end{equation}
such that
\begin{itemize}
\item  $\mathfrak{A}:=(\Omega ,{{\mathcal{F}}},{\mathbb{F}},\mathbb{P})$ is a complete filtered
probability space with a filtration ${\mathbb{F}}=\{{{\mathcal{F}}}_t:t\in [0,T]\}$ satisfying the usual conditions,
\item $W_1$ and $W_2$ are $H_1$--valued, respectively $H_2$--valued Wiener processes  over the probability space
$\mathfrak{A}$ with covariance $Q_1$ and $Q_2$;
\item $u:[0,T]\times \Omega \to H^{-1}_2(\CO)$  and $v:[0,T]\times \Omega \to L^4(\CO) 
$  are two  ${\mathbb{F}}$-progressively
measurable
processes such that the couple $(u,v)$ is a solution to the system \eqref{sys1noise} over the probability space $\mathfrak{A}$.
 \end{itemize}
\end{definition}

\begin{theorem} \label{main.martingal}
Let Assumption \ref{wiener} be satisfied. Then, for all initial conditions $(u_0,v_0)\in H_2^{-1}(\CO)\times L^4(\CO)$ satisfying Assumption
\ref{init} and for all $T>0$, there exists a martingale solution~$(\mathfrak{A},(
W_1,W_2),(u,v))$ to the system \eqref{sys1noise}
satisfying the following properties:
\begin{enumerate}[label=(\roman*)]
\item $\PP\otimes \mbox{-{Leb}-{a.s.}}$ $u(x,t)\ge 0$ and $v(x,t)\ge 0$;
  \item there exists a
positive constant~$C_1=C_1(T, u_0,v_0)$ such that
\DEQS
\EE \Big\{\sup_{0\le s\le T}|u(s)|_{L^{\gamma+1}}^{\gamma+1} +\frac \gamma2\,(\gamma+1)(\gamma) \int_0^T |u^{\gamma}(s)
\nabla u(s)|_{L^2}^2\, ds \Big\}\le C_1;
\EEQS
  \item there exists a constant $C_2=C_2(T, u_0,v_0)>0$ such  that
\DEQS
 \EE\Big[ \sup_{0\le s\le T}|u(s)|_{H^{-1}_2}^2 +\int_0^T |u(s)|^{\gamma+1}_{L^{\gamma+1}}\, ds\Big]
\le  C_2;
\EEQS
  \item there exist constants $C_3=C_3(T, u_0,v_0)$ and $C_4=C_4(T, u_0,v_0)>0$ such  that
\DEQS
\EE \|v\|_{L^{\gamma+1}(0,T;H^1_{4})}^{4}\le \EE|v_0|_{L^4} + C_3
\; &\mbox{and}& \;
\EE \|v\|_{C([0,T];L^4)}^4\le \EE|v_0|_{L^4} + C_4.
\EEQS

\end{enumerate}
\end{theorem}

In Section \ref{martingale}, we will show the existence of a martingale solution to the system \eqref{sys1noise}. In Section \ref{auxiliar}, we present several auxiliary propositions which are essential for the existence theory. 
\del{
\begin{corollary}\label{main_example}
If the conditions of Theorem \ref{main.thm} and Theorem \ref{thm path uniq} are satisfied,
then there exists a unique strong solution to system  \eqref{sys1noise.ito.new} in $C_b([0,T]; H^{-1}_2(\CO))\times C_b([0,T];  L^{2}(\CO))$.

\end{corollary}
}

\section{Existence of a martingale solution to the system \eqref{sys1noise}}\label{martingale}
As mentioned in  the introduction, we will show first the existence of a martingale solution, which is given below. Before starting with the actual proof, we first recall the drawback of the Stratonovich stochastic integral that this integral is not a martingale and the Burkholder–Davis–Gundy inequality does not hold here.
Hence, it is convenient to work the equation in It\^o form, which is presented below.
\DEQSZ\label{sys1noise.ito}
\lk\{ \barray
d{u}(t)&=& \Big( r_u \DeltaA  u^{[\gamma]} (t) 
-\chi \mbox{div} \big( u(t)\nabla v(t) -\mu_1 u(t) \big)\Big)\, dt+\sigma_u u(t) dW_1(t)
\\
d{v}(t) &=& \big(r_v \DeltaA  v(t)+\beta u(t) 
-(\alpha+\mu_2)  v(t)\Big)\, dt +\sigma_v v(t) dW_2(t),\,\quad t\in \RR_0^+.\phantom{\big|}
\earray\rk.
\EEQSZ
Adding the correction term  $\mu_1$ and $\mu_2$, we see that system \eqref{sys1noise.ito} is equivalent to \eqref{sys1noise}, where the stochastic integral is interpreted as the Stratonovich integral. The advantage of the It\^o stochastic integral is, that the It\^o integral is a local martingale and we can apply the Burkholder--Davis--Gundy inequality. For detailed discussion about the correction term and the conversion between the It\^o and Stratonovich forms, we refer the readers to Section 2.1 of \cite{EH+DM+TT_2020}. For simplicity we renamed $\alpha$; in particular, $\alpha$ corresponds to $\alpha+\mu_2$ and $\mu=\mu_1$. In this way we end up with the following system
\DEQSZ\label{sys1noise.ito.new}
\lk\{ \barray
d{u}(t)&=& \Big( r_u \DeltaA  u^{[\gamma]} (t) 
-\chi \mbox{div} \big( u(t)\nabla v(t) +\mu u(t) \big)\Big)\, dt+\sigma_u u(t) dW_1(t)
\\
d{v}(t) &=& \big(r_v \DeltaA  v(t)+\beta u(t) 
-\alpha  v(t)\Big)\, dt +\sigma_v v(t) dW_2(t),\,\quad t\in \RR_0^+.\phantom{\big|}
\earray\rk.
\EEQSZ
\noindent
Hence, due to the drawbacks of the Stratonovich integral, in the proof, we will show the existence of a solution to system \eqref{sys1noise.ito.new}.
However, by the correction term, system  \eqref{sys1noise.ito.new} is equivalent to system \eqref{sys1noise}.

\medskip

\begin{proof}[Proof of Theorem \ref{main.martingal}]
In this proof, we will show the existence of  a solution (in the sense of Definition \ref{Def:mart-sol}) to the system \eqref{sys1noise.ito.new}. The proof is split into three main steps. First, we construct an integral operator on an appropriate space and show its compactness.
Then, we  formulate a stochastic version of the Schauder--Tychonoff Fixed Point Theorem to obtain a solution. In this manner, we achieve only the existence of a martingale solution. 
\begin{steps}
\item {\bf Definitions of the underlying spaces:}
Let us consider
\begin{align}\label{def XU}
&\mathbb{X}_\mathfrak{A}=\Big\{ \xi:[0,T]\times \Omega\to H^{-1}_2(\CO)\quad  \mbox{be progressively measurable over $\mathfrak{A}$ and}
\notag\\
&\hspace{5.5cm}\EE \sup_{0\le s\le T}|\xi(s)|_{H_2^{-1}}^2
<\infty\Big\},
\end{align}
equipped with the norm
\begin{align*}
\|\xi\|_{\mathbb{X}}=\Big\{\EE
\sup_{0\le s\le T}|\xi(s)|_{H_2^{-1}}^2
\Big\}^\frac 12 .
\end{align*}
\noindent
For $\sigma\ge 1 $ and for a Banach space $E$, let us define  the collection of processes
\begin{align}\label{def mt}
\CM_{\mathfrak{A}}^{\sigma}(0,T;E)
&=\Bigg\{ \xi :[0,T] \times \Omega \rightarrow E \quad
\mbox{ be a progressively  measurable process over $\mathfrak{A}$,}
\notag\\
& \quad\mbox{ such that}\,\,\EE \Big[
\int_0^ T |\xi(s)|^\sigma _{E}\, ds\Big] < \infty
\,\Bigg\}
\end{align}
equipped with the norm 
\DEQS
\|\xi\|_{\CM^{\sigma}(0,T;E)}:=
\lk(\EE \Big[\int_0^ T  |\xi(s)|^\sigma_{E}\, ds \Big]\rk)^\frac 1\sigma .
\EEQS
Finally, for $R_1,R_2 >0$ and fixed $\gamma$, let us define the following subspace
\DEQS
\lqq{
\CX_\MA(R_1,R_2):=\bigg\{ \xi\in\BX_\MA: \quad \xi\ge 0\; \;\PP\mbox{-{Leb}-{a.s.}}\quad
} &&\\
&&\EE \Big(\sup_{0\le s\le T}|\xi(s)|_{L^{\gamma+1}}^{\gamma+1} +\frac {r_u}2\,(\gamma+1)\gamma \int_0^T \xi^{\gamma}(s)
|\nabla \xi(s)|_{L^2}^2\, ds \Big)\le R_1,\phantom{\Bigg|}
\\
&&\quad \EE\Big( \sup_{0\le s\le T}|\xi(s)|_{H^{-1}_2}^2 +\int_0^T |\xi(s)|^{\gamma+1}_{L^{\gamma+1}}\, ds\Big)\le R_2\,\,\bigg\}.
\EEQS

\medskip

\item {\bf Definition of the operator:}

Let us define the operator $\CT$ acting on $ \mathbb{X}_\MA$ as follows.
For $\eta\in \mathbb{X}_\MA$, 
 let
$\CT\eta:=u$, where $(u,v)$
solves the following system
\DEQSZ
d{u}(t)&=& \Big( r_u \DeltaA  u^{[\gamma]} (t) 
-\chi \mbox{div} \big( \eta(t)\nabla v(t)\big)+\mu u(t)\Big)\, dt+\sigma_u u(t)\, dW_1(t),\label{sysu}
\EEQSZ
and
\DEQSZ
d{v}(t) &= &\big(r_v \DeltaA  v(t)+\beta \eta(t) 
-\alpha v(t)\Big)\, dt +\sigma_v v(t)\, dW_2(t),\,\quad t\in [0,T].\phantom{\big|}\label{sysv}
\EEQSZ

First, note that due to Proposition \ref{prop existence}  the operator is well define on
$$\SubX\subset \mathbb{X}_\MA. 
$$
 In particular, for any $\eta\in \SubX$, there exist processes $u$ and $v$ solving system \eqref{sysu}--\eqref{sysv} such that $u\in  \mathbb{X}_\MA$ 
 and $\PP$-a.s.\  $v\in  L^{4}(0,T;H^2_{4}(\CO))\cap C([0,T];L^4(\CO))$.
 Due to Proposition \ref{prop conti}, there exist $R_1>0$ and $R_2>0$ such that $\mathcal{T}$ maps $\SubX$ into itself.
 Proposition \ref{continuity} gives the continuity of the operator $\CT$ from $\SubX$ into $\mathbb{X}_\MA$ and by Proposition \ref{CC} we know that the operator $\CT$ maps  $\SubX$ to a precompact set.

\item {\bf Application of the Schauder--Tychonov type Theorem:}
This Step coincide literally with Step III in our earlier paper \cite{EH+DM+TT_2020}.  Each process~$\eta\in  \mathcal{X}_\MA(R_1,R_2)$ is assigned a value~$\mathcal{T}\eta:=u$
and another value~$\Mr \eta:=v$
where~$(u,v)$ is the unique solution to system \eqref{sys1noise.ito.new}. In this way, we formulate Schauder-Tychonoff-Type Fixed Point Theorem and show that there exists an $u^\ast\in  \mathcal{X}_\MA(R_1,R_2)$ and a
corresponding $v^\ast$ solving equation
\eqref{sysv} with $\eta=u^\ast$,  and the pair $(u^\ast,v^\ast)$ is a
solution to system \eqref{sys1noise.ito.new}.

\end{steps}

\end{proof}

\section{Auxiliary Propositions specifying the properties of system \eqref{sysu}--\eqref{sysv}}\label{auxiliar}

We shortly give the exact form of the Burkholder--Davis--Gundy inequality which will be useful
during the course of analysis. Given a Wiener process ${W}$ being cylindrical on $\CH$
 over $\MA$, and a progressively measurable process $\xi\in M^2_\MA(0,T;L(\mathcal H,E))$,
let us define $\{Y(t):t\in[0,T]\}$ by
 $$
 Y(t):=\int_0^ t \xi(s)\, d {W}(s), \quad t\in[0,T].
 $$
Here, $\xi=m(u)$, where $u$ is function-valued  stochastic process and for each $t\in[0,T]$ $\xi(t)=m(u(t))$ is interpreted as a multiplication operator acting on the
elements of $\CH$, namely, $m(u):\CH\ni\psi \mapsto u\psi\in \CS'(\RR)$.
Taking the representation of $W$ by its sum
$$
W(t)=\sum_{k\in\BI}\psi_k \beta_k(t)
$$
then, in case $E$ is a Hilbert space, the Hilbert-Schmidt norm of $m(u)$ is given by
$$
|m(u)|_{L_{HS}(\CH,E)}:=\Big( \sum_{k\in\BI}|u\,\psi_k|^2_E\Big)^\frac 12 .
$$
Consequently, for any $p\ge 1$, we get for any  progressively measurable process~$\xi\in M^2_\MA(0,T;L_{HS}(\mathcal H,E))$,
\begin{equation}\label{BDG}
\EE \Big[\sup_{t \in [0, T]} |Y(t)|^p_{E}\Big] \leq C_p \, \EE \,\Big[ \int_0^T |\xi(t)|^2_{L_{HS}(\CH,E)}\, dt\Big]^\frac p2.
\end{equation}
Let $E=H^{-1}_2(\CO)$, then we have for all $\delta>0$ (see \cite[Theorem 3, p.\ 179]{runst})
\DEQSZ\label{BGI1}
\lk|m(u)\rk|_{L_{HS}(\CH,E)}^2\le \sum_{k\in\BI}|u\psi_k|^2_{H^{-1}_2}\le |u|^2_{H^\delta_2} \sum_{k\in\BI}|\psi_k|^2_{L^2}.
\EEQSZ
By the interpolation and the Young inequalities we infer that for all $\ep>0$, there exists a constant $C=C(\ep)>0$ such that  
\DEQSZ\label{BGI1int}
\qquad \lk|m(u)\rk|_{L_{HS}(\CH,E)}^2
\le \ep |u|^2_{H^1_2}+ C(\ep)|u|^2_{H^{-1}_2} \sum_{k\in\BI}|\psi_k|^2_{H^\delta_2 }.
\EEQSZ
For $E=L^2(\CO)$ we know for any $\delta>2$, $q$ with $\frac 1q+\frac 1\delta\ge \frac 12$, 
\DEQS
\qquad \lk|m(u)\rk|_{L_{HS}(\CH,E)}^2\le \sum_{k\in\BI}|u\psi_k|^2_{L^2}\le  \sum_{k\in\BI}|u|^2_{L^q}|\psi_k|^2_{L^\delta }.
\EEQS 
Again, by the interpolation and the Young inequalities we declare that for all $\ep>0$, there exists a constant $C=C(\ep)>0$ such that  
\DEQSZ\label{BGI2}
\qquad \lk|m(u)\rk|_{L_{HS}(\CH,E)}^2\le \ep |u|^2_{H^1_2}+ C(\ep)|u|^2_{H^{-1}_2} \sum_{k\in\BI}|\psi_k|^2_{L^\delta }.
\EEQSZ
In case, $E$ is a Banach space, the Banach space  $\gamma(\CH,E)$ is defined as  the completion of
$\CH\otimes E$ with respect to the norm
\DEQSZ\label{aaa}
\bigg( \Big|\sum_{k=1}^{N} \tilde \psi_k\otimes x_k\Big|^2\bigg)^\frac 12 &:=& \bigg(\Big| \EE \sum_{k=1}^{N} r_kx_k\Big|^2\bigg)^\frac 12,
\EEQSZ
for all finite sequences $x_1,\cdots, x_N \in E$
and $\{\tilde{\psi}_k:k=1,\ldots,N\}$ are assumed to be orthonormal in $\CH$. Here~$(r_k)_{k \geq 1}$ is a Rademacher sequence; for more resources we cite \cite{vannnervensurvey}. Throughout the paper, let us denote the norm introduced in \eqref{aaa} by $\gamma(\CH,E)$.
Similar to \eqref{BDG}, for any $p\ge 1$ we get for any  progressively measurable process $\xi\in M^2_\MA(0,T;\gamma(\CH,E))$,
\begin{equation*}\label{BDGBanach}
\EE \Big[\sup_{t \in [0, T]} |Y(t)|^p_{E}\Big] \leq C_p \, \EE \,\Big[ \int_0^T |\xi(t)|^2_{\gamma(\CH,E)}\, dt\Big]^\frac p2.
\end{equation*}
For further details we refer to the survey on $\gamma$--radonifying operator \cite{vannnervensurvey}.
Let $E= L^4(\CO)$ and $v\in H^{-1}_{4}(\CO)\cap H^1_4(\CO)$ then 
we get by the H\"older inequality and the Sobolev embedding
\begin{align}\label{BGI3}
\lk|m(u)\rk|_{\gamma(\CH,E)}^2 &\le \sum_{k\in\BI}|u\psi_k|^2_E\le C|u|^2_{L^{20}} \sum_{k\in\BI}|\psi_k|^2_{L^5}\nonumber
\\
& \le C\, |u|_{H^\frac 25_4} ^2 \sum_{k\in\BI}|\psi_k|^2_{L^5}\le \ep|u|^2_{H^1_4}+C|u|_{H^{-1}_4}^2\sum_{k\in\BI}|\psi_k|^2_{L^5}.
\end{align}
Let $E=L^{\gamma+1}(\CO)$. Then we get by similar calculations as above for any $0<\delta<\frac 1\gamma$
\DEQS
\lk|m(u)\rk|_{\gamma(\CH,E)}^2\le \sum_{k\in\BI}|u\psi_k|^2_E\le C|u|_{H^{\frac \delta2}_{\gamma+1}} \sum_{k\in\BI}|\psi_k|^2_{L^{\gamma+1}}
\le \ep|u|^2_{H^{\delta}_{\gamma+1}}
+C\, |u|^2_{L^{\gamma+1}}  \sum_{k\in\BI}|\psi_k|^2_{L^{\gamma+1}}.
\EEQS
Let $E= H^\frac 2{\gamma+1}_{\gamma+1}(\CO)$ and $u\in H^{\frac {\gamma+3}{\gamma+1}}_{\gamma+1}(\CO)\cap H^{-\frac {\gamma-1}{\gamma+1}}_{\gamma+1}(\CO)$ then we get
\DEQSZ\label{BGI5}
\lk|m(u)\rk|_{\gamma(\CH,E)}^2
\le \ep |u|_{H^{\frac {\gamma+3}{\gamma+1}}_{\gamma+1}}^2 +C\, |u|_{H^{-\frac {\gamma-1}{\gamma+1}}_{\gamma+1}}^2 \sum_{k\in\BI}|\psi_k|^2_{H^{2\frac {\gamma}{\gamma+1}}_{2}},
\EEQSZ
and finally for $E=L^{\gamma+1}(\CO)$ we achieve in an alternate fashion
\DEQSZ\label{BGIgamma}
\lk|m(u)\rk|_{\gamma(\CH,E)}^2\le \sum_{k\in\BI}|u\psi_k|^2_E\le C|u|_{L^{\gamma+1}} \sum_{k\in\BI}|\psi_k|^2_{H^\frac {\gamma-1}{\gamma+1}_{\gamma+1}}
\le C\, |u|_{L^{\gamma+1}}  \sum_{k\in\BI}|\psi_k|^2_{H^{2\frac {\gamma}{\gamma+1}}_{2}}.
\EEQSZ
At once, we compare Lemma 3.5 and Lemma 3.6 in \cite{gess2020} and state the following technical Proposition in our settings.
\begin{tproposition}\label{runst1}
Let $\gamma>1.$ Then, for any $\theta$ with $0<\theta<\frac 1 \gamma,$ there exists a constant $C>0$ such that for all $w$ with $|w|^\gamma\in H^1_2(\CO),$
		$$	
C		 |w|_{H^\theta_{2\gamma}}^\gamma \le ||w|^\gamma|_{H^1_2}.
		$$
Furthermore, for any $\eta$ with 
$\int_0^T |\eta^{\gamma-1}(s)\nabla \eta(s)|_{L^2}^2\, ds<\infty,$
there holds
\DEQS
\|\eta\|_{L^{2\gamma}(0,T;H^{\theta}_{2\gamma})}^{2\gamma} \le \int_0^T |\eta^{\gamma-1}(s)\nabla \eta(s)|_{L^2}^2\, ds.
\EEQS
\end{tproposition}
\begin{proof}
We know by Runst and Sickel, \cite[p.~365]{runst}  that for any  $p\in(1,\infty)$, $s\in(0,1)$, $\ep\in(0,\infty)$, and $\mu\in(0,1)$
	$$
| |w|^\mu |_{H^{s-\ep}_p}  = 	| |w|^\mu\mid _{F^{s-\ep}_{\frac p\mu,2}} \le 	C| |w|^\mu\mid_{F_{\frac p\mu,\frac 2\mu}^s} \le C	| w\mid _ {F_{p, 2 }^\frac{s}{\mu}}^\mu =C | w|_{H^{s}_p}^\mu, \quad w\in H^\frac{s}{\mu}_p(\CO).
	$$
Reformulating, we have that  for any  $p\in(1,\infty)$, $\mu,s\in(0,1)$ and $\ep\in(0,\infty)$ there exists a constant $C>0$ such that
	\DEQSZ
	\lk| |w|^\mu \rk|_{H^{s\mu-\ep}_{\frac p\mu}}\le  C |w|_{H^s_p}^\mu , \quad w\in H^s_p(\CO)\label{from_runst}
	.
		\EEQSZ
		From \eqref{from_runst} we know that  for any  $\gamma>1$, $\theta\in(0,\frac 1\gamma)$, $\ep>0$, and $p> \gamma$, there exists a constant $C>0$ such that
		\DEQS
	 |w|_{H^\theta_p}=||w^{\gamma} |^\frac 1\gamma |_{H^\theta_p}\leq C ||w|^\gamma|^{\frac 1\gamma}_{H^{\gamma(\theta +\ep)}_{\frac p\gamma}}
		.
		\EEQS
In particular, for any  $\theta<\frac 1 \gamma$ and $p=2\gamma$, there exists a constant $C>0$ such that 
for all $w$ with $|w|^\gamma\in H^1_2(\CO),$
		$$	
C		 |w|_{H^\theta_{2\gamma}}^\gamma \le ||w|^\gamma|_{H^1_2}.
		$$
		This proves the first part of the proposition. Again, since we know that
$$\int_0^ T |\eta^{\gamma-1}(s)\nabla \eta(s) |_{L^2}^2\, ds=\int_0^ T |\nabla \eta ^{\gamma}(s)|_{L^2}^2\, ds\sim \int_0^ T |\eta ^{\gamma}(s)|_{H^1_2}^2\, ds\le R_2<\infty,$$
hence, for any $0<\theta<\frac 1\gamma$ and $p=2\gamma$
$$\int_0^ T |\eta(s)|_{H^\theta_{2\gamma}}^{2\gamma}\, ds\le C \int_0^ T |\eta ^{\gamma-1}(s)\nabla \eta(s) |_{L^2}^2\, ds
.$$
This finishes the proof of the proposition.
\end{proof}
With these estimates at hand, we proceed to the next proposition to declare that the mapping $\CT$ on $\BX_\MA$
 is well defined.
\begin{proposition}\label{prop existence}
For any $R_1,R_2>0$ and $\eta\in\CX_\MA(R_1,R_2)$, there exists a unique pair $(u,v)$ of solutions to the systems \eqref{sysu}--\eqref{sysv} such that the following holds:
\DEQS
\EE\Big[
\sup_{0\le s\le T}|u(s)|_{H^{-1}_2}^2
+\int_0^T |u(s)|^{\gamma+1}_{L^{\gamma+1}}\, ds\Big]<\infty.
\EEQS
There exist two constants $C_1,C_2>0$ such that 
\DEQS
\EE \|v\|_{L^4(0,T;H^1_4)} &\le&  \EE|v_0|^{4}_{L^4}+C\, \EE \|\eta \|_{L^2(0,T;H^{-1}_2)}^2,
\EEQS
and
\DEQS
\EE \|v\|_{C([0,T];H^0_4)}^4 &\le&  \EE|v_0|^{4}_{L^4}+C\, \EE \|\eta \|_{L^2(0,T;H^{-1}_2)}^2.
\EEQS
If, in addition, $\EE|v_0|^{\gamma+1}_{H^{\frac{2\gamma}{\gamma+1}}_{\gamma+1}}<\infty$ , then there exist two constants $C_1,C_2>0$ such that 
\DEQSZ
\EE\Big[ \sup_{0\le s\le T}|v(s)|_{H^1_2}^{\gamma+1}\Big]
 &\le&  \EE|v_0|^{\gamma+1}_{H^1_{\gamma+1}}
+C_1\EE\|\eta\|_{L^{\gamma+1}(0,T;H^{-\frac{\gamma-1}{\gamma+1}}_{\gamma+1})}^{\gamma+1},\label{esti reg v1}
\EEQSZ
and
\DEQSZ
  \EE \|v\|^{\gamma+1}_{L(0,T;H^{\frac {\gamma+3}{\gamma+1}}_{\gamma+1})} &\le& \EE|v_0|^{\gamma+1}_{H^1_{\gamma+1}}+
C_2\, \EE \|\eta\|^{\gamma+1}_{L^{\gamma+1}(0,T;H^{-\frac{\gamma-1}{\gamma+1}}_{\gamma+1})}.\label{esti reg v2}
\EEQSZ

\end{proposition}
\begin{remark}\label{remark11}
If $\gamma\ge 3$, then $\frac {\gamma+1}{\gamma-1}\le 2$ and we obtain
\DEQS\label{reg_v}
\EE \|v\|_{L^\frac {\gamma+1}{\gamma-1}(0,T;H^1_{\frac {\gamma+1}\gamma})} ^\frac {\gamma+1}{\gamma-1}
\le C\,\lk( \EE \|v\|_{L^{2}(0,T;H^1_2)} ^{2}\rk)^\frac {\gamma+1}{2(\gamma-1)}\le C\,\lk( \EE \|v\|_{L^{\gamma+1}(0,T;H^1_{\gamma+1})} ^{\gamma+1}\rk)^\frac {1}{\gamma-1}.
\EEQS
\end{remark}

\begin{proof}
The proof consists of two steps. First in Claim \ref{existencev}, we will investigate equation \eqref{sysv} and show that there exists a unique solution $v$ to \eqref{sysv} and
specify the integrability and regularity properties of $v$. Then as a second step in Claim \ref{existence}, we will give the existence of the unique solution $u$ to \eqref{sysu} and
show that $u\in\CM_\MA^{\gamma+1}(0,T;L^{\gamma+1}(\CO))$.
\begin{claim}\label{existencev}
Let us consider the equation
\DEQSZ\label{eq v 1}
d{v}(t) &=& \Big(r_v \DeltaA  v(t)+\beta u(t) 
-\alpha  v(t)\Big)\, dt +\sigma_v v(t) dW_2(t),\,\quad t\in \RR_0^+,\phantom{\big|}
\EEQSZ
with initial condition $v(0)=v_0$, where $v_0$ is a $\CF_0$--measurable random data.
\begin{enumerate}
\item
If  $\EE|v_0|_{L^2}^2<\infty$ and $\eta \in \CM^{2}_{\mathfrak{A}}(0,T;H^{-1}_2(\CO))$, then a unique solution $v$ to \eqref{eq v 1} exists and
there exists a constant $C>0$ such that
\DEQS
\EE \|v\|_{L^2(0,T;H^1_2)}\le \EE|v_0|_{L^2}^2 +C \EE \|\eta \|_{L^2(0,T;H^{-1}_2)}^2,
\EEQS
and
\DEQS
\EE \|v\|_{C([0,T];L^2)}^2 &\le&  \EE|v_0|^{2}_{L^2}+C\, \EE \|\eta \|_{L^2(0,T;H^{-1}_2)}^2.
\EEQS
\item
If  $\EE|v_0|_{L^4}^2<\infty$ and $\eta \in \CM^{2}_{\mathfrak{A}}(0,T;H^{-1}_4(\CO))$, then a unique solution $v$ to \eqref{eq v 1} exists and
there exists a constant $C>0$ such that
\DEQS
\EE \|v\|^2_{L^4(0,T;H^1_4)} &\le&  \EE|v_0|^{2}_{L^4}+C\, \EE \|\eta \|_{L^2(0,T;H^{-1}_4)}^2,
\EEQS
and
\DEQS
\EE \|v\|_{C([0,T];L^4)}^2 &\le&  \EE|v_0|^{2}_{L^4}+C\, \EE \|\eta \|_{L^2(0,T;H^{-1}_4)}^2.
\EEQS

\item
If  $ \EE|v_0|^{\gamma+1}_{H^1_{\gamma+1}}<\infty$ and $\eta\in \CM^{\gamma+1}_{\mathfrak{A}}(0,T;H^{-\frac{\gamma-1}{\gamma+1}}_{\gamma+1}(\CO))$, then a unique solution $v$ to \eqref{eq v 1} exists and
there exist constants $C_1,C_2>0$ such that
\DEQS
\EE\Big[ \sup_{0\le s\le T}|v(s)|_{H^1_{\gamma+1}}^{\gamma+1}\Big]
\le \EE|v_0|^{\gamma+1}_{H^1_{\gamma+1}}
+C_1\EE\|\eta\|_{L^{\gamma+1}(0,T;H^{-\frac{\gamma-1}{\gamma+1}}_{\gamma+1})}^{\gamma+1}.
\EEQS
and
\DEQS  \EE \|v\|^{\gamma+1}_{L^{\gamma+1}(0,T;H^{\frac {\gamma+3}{\gamma+1}}_{\gamma+1})}&\le &\EE|v_0|^{\gamma+1}_{H^1_{\gamma+1}}+
C_2\, \EE \|\eta\|^{\gamma+1}_{L^{\gamma+1}(0,T;H^{-\frac{\gamma-1}{\gamma+1}}_{\gamma+1})}.
\EEQS

\item
If $\EE|v_0|^{\gamma+1}_{H^{\frac{2\gamma}{\gamma+1}}_{\gamma+1}}<\infty$ and $\eta\in \CM^{\gamma+1}_{\mathfrak{A}}(0,T;L^{\gamma+1}(\CO))$,
then a unique solution $v$ to \eqref{eq v 1} exists and
there exist constants $C_1,C_2>0$ such that
\DEQS
\EE \|v\|_{L^{\gamma+1}(0,T;H^2_{\gamma+1})}^{\gamma+1}\le\EE|v_0|_{H^{\frac {2 \gamma}{\gamma+1}}_{\gamma+1}} 
^{\gamma+1}+C_1\EE \|\eta\|_{L^{\gamma+1}(0,T;L^{\gamma+1})}^{\gamma+1},
\EEQS
and
\DEQS
\EE \|v\|_{C([0,T];H^{\frac {2 \gamma}{\gamma+1}} _{\gamma+1})}^{\gamma+1}\le \EE|v_0|_{  H_{\gamma+1}^{\frac {2 \gamma}{\gamma+1} }   }
^{\gamma+1}+C_2\EE \|\eta\|_{L^{\gamma+1}(0,T;L^{\gamma+1})}^{\gamma+1}.
\EEQS\end{enumerate}
\end{claim}
\del{
\begin{remark}\label{remark11}
If $\gamma\ge 2$, then $\frac {\gamma+1}{\gamma-1}\le \gamma+1$ and we obtain
\DEQSZ\label{reg_v}
\EE \|v\|_{L^\frac {\gamma+1}{\gamma-1}(0,T;H^1_{\frac {\gamma+1}\gamma})} ^\frac {\gamma+1}{\gamma-1}
\le C\,\lk( \EE \|v\|_{L^{\gamma+1}(0,T;H^1_2)} ^{\gamma+1}\rk)^\frac 1{\gamma-1}\le C\,\lk( \EE \|v\|_{L^{\gamma+1}(0,T;H^1_{\gamma+1})} ^{\gamma+1}\rk)^\frac 1{\gamma-1}.
\EEQSZ

\end{remark}}

\begin{proof}[Proof of Claim \ref{existencev}:]
By Example 3.2 (1)-(4) \cite{vanNeerven1} we know that the Laplace operator with Neumann boundary condition has a bounded $H^\infty$-calculus on $H^{-1}_2(\CO)$, $L^2(\CO)$ and $L^{\gamma+1}(\CO)$.
The item (a) to (d) in this claim is an application of Theorem 4.5 in \cite{vanNeerven1}.
To show item (a), let us put in Theorem 4.5 $p=2$, $X_0=H^{-1}_2(\CO)$, $X_1=H^1_2(\CO)$, $B(v)[\psi]:=v\,\psi$, $\psi\in \mathcal H$, $F(v)=\alpha v$, (and incorporating $f(t):=\eta(t)$).
Note, that the Hilbert-Schmidt norm of $B$  is evaluated in \eqref{BGI2} and due to Assumption \ref{wiener}, we know $B$ satisfies Hypothesis (HB) (see \cite[p.\ 1384]{vanNeerven1}).
It follows that
there exists a constant $C>0$ such that
\DEQS  \EE \|v\|^{2}_{L(0,T;H^1_{2})}&\le & \EE|v_0|^2_{L^2}+
C\, \EE \|\eta\|^2_{L(0,T;H^{-1}_2)}.
\EEQS

To show item (b), let us substitute in Theorem 4.5 $p=2$, $X_0=H^{-1}_4(\CO)$, $X_1=H^1_4(\CO)$, $B(v)[\psi]:=v\,\psi$, $\psi\in \mathcal H$, $F(v)=\alpha v$, (and incorporating $f(t):=\eta(t)$).
Note, that $|B(u)|_{\gamma(\mathcal H,L^4)}$  is evaluated in \eqref{BGI3} and due to Assumption \ref{wiener}, we infer that $B$ satisfies Hypothesis (HB) (see \cite[p.\ 1384]{vanNeerven1}).
It is elementary to see that
there exists a constant $C>0$ such that
\DEQS  \EE \|v\|^{2}_{L^2(0,T;H^1_{4})}&\le &\EE|v_0|^2_{L^4}+
C\, \EE \|\eta\|^2_{L(0,T;H^{-1}_4)},
\EEQS
and
\DEQS
\EE \|v\|_{C([0,T];L^4)}^2 &\le&  \EE|v_0|^{2}_{L^4}+C\, \EE \|\eta \|_{L^2(0,T;H^{-1}_4)}^2.
\EEQS

To show item (c), let us continue using Theorem 4.5 with $p=\gamma+1$,
$$X_0=H^{-\frac{\gamma-1}{\gamma+1}}_{\gamma+1}(\CO),\quad X_1=H^{\frac {\gamma+3}{\gamma+1}}_{\gamma+1}(\CO), \quad  X_{\frac 12 }=H^{\frac {2}{\gamma+1}}_{\gamma+1}(\CO),\quad \mbox{and}\quad
X_{1-\frac 1{p},p}=H^1_{\gamma+1}(\CO).
$$
Furthermore, let us put $B(v)[\psi]:=v\,\psi$, $\psi\in \mathcal H$, $F(v)=\alpha v$, (and incorporating $f(t):=\eta(t)$).
Note, that $|B(u)|_{\gamma(\mathcal H,X_{\frac 12})}$  is evaluated in \eqref{BGI5} and owing to Assumption \ref{wiener}, we know $B$ satisfies Hypothesis (HB) (see \cite[p.\ 1384]{vanNeerven1}).
Consequently,
there exists a constant $C>0$ such that
\DEQS  \EE \|v\|^{\gamma+1}_{C([0,T];H^1_{\gamma+1})}&\le &\EE|v_0|^{\gamma+1}_{H^1_{\gamma+1}}+
C\, \EE \|\eta\|^{\gamma+1}_{L^{\gamma+1}(0,T;X_0)},
\EEQS
and
\DEQS  \EE \|v\|^{\gamma+1}_{L^{\gamma+1}(0,T;H^{\frac {\gamma+3}{\gamma+1}}_{\gamma+1})}&\le &\EE|v_0|^{\gamma+1}_{H^1_{\gamma+1}}+
C\, \EE \|\eta\|^{\gamma+1}_{L^{\gamma+1}(0,T;X_0)}.
\EEQS

Finally, to show item (d), let us continue Theorem 4.5 with $p=\gamma+1$,
$$X_0=L^{\gamma+1}(\CO),\quad X_1=H^{2}_{\gamma+1}(\CO), \quad  X_{\frac 12 }=H^{1}_{\gamma+1}(\CO),\quad \mbox{and}\quad
X_{1-\frac 1{p},p}=H^{\frac {2\gamma}{\gamma+1}}_{\gamma+1}(\CO).
$$
Furthermore, we again incorporate $B(v)[\psi]:=v\,\psi$, $\psi\in \mathcal{H}$, $F(v)=\alpha v$, (and take $f(t):=\eta(t)$).
Once again we note that $|B(u)|_{\gamma(\mathcal H,X_{\frac 12})}$  is evaluated in \eqref{BGI5} and $B$ satisfies Hypothesis (HB) (see \cite[p.\ 1384]{vanNeerven1}).
It follows that
there exists a constant $C>0$ such that (d) holds.

\del{We note that the operator $A_1=r_v A-\alpha I$ with $D(A_1)=H^{2}_{\gamma+1}(\CO)$ generates an analytic semigroup in $L^{\gamma+1}(\CO)$. Hence equation \eqref{sysv} has a solution which can be written as
\DEQS
v(t)&=& e^{t(r_v A -\alpha I)}v_0+\beta \int_0^ t e^{(t-s)(r_vA -\alpha I)} \eta(s)\, ds+\int_0^ t e^{(t-s)(r_v A-\alpha I)}\sigma_v\, v(s)\,dW_2(s).
\EEQS
Using the interpolation results (see Theorem 6.2.4 \cite{bergh}) and since $\frac{2 \gamma}{\gamma+1} \notin \NN$, we observe that
\begin{align}
(L^{\gamma+1}(\CO),H^2_{\gamma+1}(\CO))_{{1-\frac{1}{\gamma+1}},\gamma+1}
=H^{\frac{2\gamma}{\gamma+1}}_{\gamma+1}(\CO).
\end{align}
For $v_0 \in H^{\frac{2\gamma}{\gamma+1}}_{\gamma+1}(\CO)$ and using Theorem 4.5 of \cite{vannnerven12} we have estimates \eqref{esti reg v1} and \eqref{esti reg v2}.
}
\end{proof}

\begin{claim}\label{existence}
For $\gamma>1$, any $R_1,R_2>0$ and $\eta\in\CX_\MA(R_1,R_2)$, there exists a unique pair $(u,v)$ of solutions to the systems \eqref{sysu}--\eqref{sysv} such that
\begin{align}\label{esti uv}
\EE\Big[
\sup_{0\le s\le T}|u(s)|_{H^{-1}_2}^2
+\int_0^T |u(s)|^{\gamma+1}_{L^{\gamma+1}}\, ds\Big]<\infty.
\end{align}
\end{claim}

\begin{proof}[Proof of Claim \ref{existence}:]
Let us fix $R_1,R_2>0$ and $\eta\in \CX_\MA(R_1,R_2)$. We know by Claim \ref{existencev} that
there exists unique $v\in C([0,T];H^{\frac {2\gamma}{\gamma+1}}(\CO))\,\, \PP$--a.s to \eqref{sysv} and $v$ satisfies the estimate \eqref{reg_v}. We now show that there exists a
solution $u$ to \eqref{sysu} by
verifying  the assumptions of Theorem 5.1.3 in \cite{weiroeckner} for
$\gamma>1$. Let us consider the Gelfand triple
\[
V\subset \CH\cong \CH^{\ast}\subset V^{\ast},
\]
with $\CH:=H_{2}^{-1}(\CO)$, the dual space $\CH^\ast$ of ${H}_{2}^{1}(\CO)$
(corresponding to Neumann boundary conditions) and set $V:=L^{p}(\CO)$ for some fix $p=\gamma+1$ and its conjugate $ p'=\frac{\gamma+1}{\gamma}$.
Let $V^{\ast}=\{
u \in \mathcal{D}(\CO):u=-\DeltaA  v,\,\, v \in L^{p'}
\}$.
The duality $_{V^{\ast}}\langle \cdot, \cdot \rangle_V$ is defined as
\begin{align*}
_{V^{\ast}}\langle u,w \rangle_V
=\int_{\CO} ((-\DeltaA )^{-1}u)(x)\, w(x)\, dx
.
\end{align*}
We set
\[
\mathcal A(t,u,\omega):=\mathcal A_{v}(u):=r_u\DeltaA  u^{[\gamma]}-\chi\di(\eta(t,\omega) \nabla v(t,\omega)).
\]
Firstly, let us investigate the process
$$
\xi:[0,T]\times \Omega \ni (t,\omega)\mapsto \Div (\eta(t,\omega)\nabla v(t,\omega)),
$$
where $v$ solves \eqref{sysv}. We note that
\DEQS
_{V^{\ast}}\langle \xi(t,\omega),w \rangle_V &=& \int_\CO (-\Delta)^{-1}\lk( \Div (\eta(t,\omega)\nabla v(t,\omega)\rk) w(x)\, dx.
\EEQS
Now using the fact that $L^1(\CO)\hookrightarrow H^{-1}_{\frac {\gamma+1}{\gamma}}(\CO)$  for  $d=1,2$ and the Young inequality, we see that for any $\ep_1,\ep_2>0$, there exists a constant $C=C(\ep_1,\ep_2)>0$ such that
\DEQSZ\label{esti vv.sta}
\lqq{ \lk|_{V^{\ast}}\langle \xi(t,\omega),w \rangle_V \rk|
\le  \lk|  \eta(t,\omega)\nabla v(t,\omega)\rk|_{H^{-1}_\frac {\gamma+1}{\gamma}} |w|_{L^{\gamma+1}}\le  \lk|  \eta(t,\omega)\nabla v(t,\omega)\rk|_{L^1} |w|_{L^{\gamma+1}}}
\nonumber\\
&\le &
\lk|\eta(t,\omega)\rk|_{L^{\gamma+1}} \lk| \nabla v(t,\omega)\rk|_{L^{\frac {\gamma+1}{\gamma}}} |w|_{L^{\gamma+1}}
\le \ep_1 \lk|\eta(t,\omega)\rk|_{L^{\gamma+1}}^{\gamma+1} + C\lk| \nabla v(t,\omega)\rk|^\frac {\gamma+1}{\gamma-1}_{L^{\frac {\gamma+1}{\gamma}}} +\ep_2|w|_{L^{\gamma+1}}^{\gamma+1}
.
\EEQSZ
Next, we will show that for fixed $t\in[0,T]$ and fixed $\omega\in\Omega$, the operator
$\mathcal A_{v}:V \rightarrow  V^{\ast}$ is indeed a bounded operator.
In fact, using \eqref{esti vv.sta} we obtain,
{\renewcommand{\|}{|}
\DEQSZ\label{eq:A-bounded}
\lqq{ \left|\mathbin{_{V^{\ast}}\langle \mathcal A_{v}(u),w\rangle_{V}}\right|
=\left|-\int_{\CO}\left[r_u u^{[\gamma]}w-\chi (-\Delta)^{-1}\di (\eta(t,\omega)\nabla v(t,\omega))w\right]\,dx\right|
}
&&
\\
\nonumber & \le&
 r_u\|u\|_{L^{\gamma+1}}^{\gamma} \|w\|_{L^{\gamma+1}}
 +\chi \|(-\Delta)^{-1} \di (\eta(t,\omega)\nabla v(t,\omega))\|_{L^{\frac{\gamma+1}{\gamma}}}\|w\|_{L^{\gamma+1}}
\\
\nonumber & \le&
 r_u\|u\|_{L^{\gamma+1}}^{\gamma}
\|w\|_{L^{\gamma+1}}
+\chi \| \eta(t,\omega)\|_{L^{\gamma+1}} \|\nabla v(t,\omega)\|_{L^{\frac {\gamma+1}\gamma}}\|w\|_{L^{\gamma+1}}.
\EEQSZ
}
To show the existence of solution $u$ to \eqref{sysu}, we will verify Hypothesis (H1), (H2$^{\prime}$), (H3), and  (H4$^{\prime}$) of \cite[Theorem 5.1.3]{weiroeckner}.
In this way, there exists a pair of solutions $(u,v)$ solving \eqref{sysu}-\eqref{sysv}.

\medskip

\paragraph{Verification of (H1) (Hemicontinuity).}
 Let $\lambda\in\RR$. We need to show that
$$
\lambda\mapsto {_{V^{\ast}}\langle \mathcal A_v(u_{1}+\lambda u_{2}),w\rangle_{V}}
$$
 is continuous on $\RR$, for any $u_{1},u_{2},w\in V$, $t\in[0,T]$,
$\omega\in\Omega$.
As the map $\lambda \mapsto (u_{1}+\lambda u_{2})^{[\gamma]}$ is continuous, by the Dominated Convergence Theorem $
\lambda\mapsto \int_{\CO}(u_{1}+\lambda u_{2})^{[\gamma]}w dx $
is continuous. Hence, $
\lambda\mapsto {_{V^{\ast}}\langle \mathcal A_v(u_{1}+\lambda u_{2}),w\rangle_{V}}$ is continuous.

\medskip

\paragraph{Verification of (H2$^{\prime}$) (Local monotonicity).}
Let $u,w\in L^{\gamma+1}(\CO)$, $t\in[0,T]$, $\omega\in\Omega$. Then, we get
\begin{align}
&  \mathbin{_{V^{\ast}}\langle \mathcal A_{v}(u)-\mathcal A_{v}(w),u-w\rangle_{V}}+|\sigma_1(u)-\sigma_1(w)|^2_{\CL_{2}(H_1,\CH)}
\notag\\
&\le  -\int_{\CO}\left[r_u(u^{[\gamma]}-w^{[\gamma]})(u-w)\right]\,dx+C|u-w|_{\CH}^{2}
\notag\\
&\le  C r_u |u-w|_{L^{\gamma+1}}^2
\Big(|u|_{L^{\gamma+1}}^{\gamma-1}
+ |w|_{L^{\gamma+1}}^{\gamma-1}
\Big)+C|u-w|_{\CH}^{2}
\notag\\
&\le  \Big[\Big(
C r_u |u|_{L^{\gamma+1}}^{\gamma-1}
+ |w|_{L^{\gamma+1}}^{\gamma-1}
\Big)+C\Big]
|u-w|_{\CH}^{2}
.
\end{align}
Hence,  (H2$^{\prime}$) of \cite[Theorem 5.1.3]{weiroeckner} holds.

\medskip

\paragraph{Verification of (H3) (Coercivity).}
Let $u,w\in L^{\gamma+1}(\CO)$, $t\in[0,T]$, $\omega\in\Omega$. Then, by substituting $\xi$ and taking $\ep_1=\frac 14$ and $\ep_2=\frac {r_u}4$ in \eqref{esti vv.sta} we achieve
\DEQS
\lqq{   \mathbin{_{V^{\ast}}\langle \mathcal A_{v}(u),u \rangle_{V}}+|\sigma_1(u)|^2_{\CL_{2}(H_1,\CH)}}
\\
&=&  - r_u \int_{\CO} u^{[\gamma]}u
-\chi \int_{\CO}\lk( (-\Delta)^{-1} \xi\rk)(x)\, u(x)\,dx+|\sigma_1(u)|^2_{\CL_{2}(H_1,\CH)}
\notag
\\
& \leq& -\frac 34r_u |u|_{L^{\gamma+1}}^{\gamma+1}
+C \chi \lk| \nabla v(t,\omega)\rk|^\frac {\gamma+1}{\gamma-1}_{L^{\frac {\gamma+1}{\gamma}}}
+
 \frac 14 \lk|\eta(t,\omega)\rk|_{L^{\gamma+1}}^{\gamma+1} +C |u|_{\CH}^{2}.
\EEQS
Owing to Remark \ref{remark11} and the assumptions on $\eta$ it follows that a non-negative adapted process
$$
f(t):= C \chi \lk| \nabla v(t,\omega)\rk|^\frac {\gamma+1}{\gamma-1}_{L^{\frac {\gamma+1}{\gamma}}}
+
 \ep_1 \lk|\eta(t,\omega)\rk|_{L^{\gamma+1}}^{\gamma+1}
$$
belongs to  $\CM^1_\MA(0,T;\RR)$. This proves (H3) of \cite[Theorem 5.1.3]{weiroeckner}.

\medskip

\paragraph{Verification of (H4$^{\prime}$) (Growth).}
 Let $u\in L^{\gamma+1}(\CO)$, $t\in[0,T]$, $\omega\in\Omega$.
Due to assumption $\eta\in \CM^{\gamma+1}_\MA(0,T;L^{\gamma+1})$ and Remark \ref{remark11}, a non-negative adapted process $f$ with
$$f(t):=\Big(\ep_1 \lk|\eta(t,\omega)\rk|_{L^{\gamma+1}}^{\gamma+1} + C\lk| \nabla v(t,\omega)\rk|^\frac {\gamma+1}{\gamma-1}_{L^{\frac {\gamma+1}{\gamma}}}\Big)\in \CM^{1}_\MA(0,T;\RR).$$
Again using \eqref{eq:A-bounded}  we have
\DEQS
 \left|\mathcal A_{v}(u)\right|_{V^\ast}^{\frac {\gamma+1}\gamma}
 & \le&
 r_u|u|_{L^{\gamma+1}}^{\gamma+1}
+\chi\lk(  | \eta(t,\omega)|_{L^{\gamma+1}} |\nabla v(t,\omega)|_{L^{\frac {\gamma+1}\gamma}}\rk)^\frac {\gamma+1}\gamma.
\EEQS
By the Young inequality for $p=\gamma$ and $p'=\frac {\gamma}{\gamma-1}$, we get
\DEQS
 \left|\mathcal A_{v}(u)\right|_{V^\ast}^{\frac {\gamma+1}\gamma}
 & \le&
 r_u|u|_{L^{\gamma+1}}^{\gamma+1}
+\frac 14  | \eta(t,\omega)|_{L^{\gamma+1}}^{\gamma+1}+\chi C\, |\nabla v(t,\omega)|_{L^{\frac {\gamma+1}\gamma}}^\frac {\gamma+1}{\gamma-1}.
\EEQS
Then,
(H4$^{\prime}$) of \cite[Theorem 5.1.3]{weiroeckner} holds (with $\alpha=\gamma+1$,
$\beta=0$).
%


\medskip

In this way we have shown that the Hypothesis (H1), (H2$^{\prime}$), (H3), and  (H4$^{\prime}$) of \cite[Theorem 5.1.3]{weiroeckner} are satisfied. Therefore, by an application of \cite[Theorem 5.1.3]{weiroeckner}, the existence of $u$ is guaranteed
satisfying
$$
\EE\Big[\sup_{0\le s\le T}|u(t)|_{H^{-1}_2}^2\Big]+\EE\Big[\int_0^ T |u(t)|_{L^{\gamma+1}}^{\gamma+1}\, dt\Big]<\infty.
$$

\end{proof}

From Claim \ref{existencev} we have shown that, if $\gamma>1$, for given $\eta$, there exists an unique solution $v$ to \eqref{sysv} such that
\DEQSZ \label{exholds}
\EE\Big[\int_0^T  |\nabla v(s)|_{L^\frac {\gamma+1}\gamma}^ \frac {\gamma+1}{\gamma-1}\Big]<\infty,
\EEQSZ
 and by Claim \ref{existence} we have
shown that if $v$ satisfies \eqref{exholds}, then
there exists a unique solution $u$ to \eqref{sysu}. Hence, a unique pair $(u,v)$ of solutions  to  \eqref{sysu}--\eqref{sysv}
exists such that
$u\in\BX_{\mathfrak A}$ and
\begin{align}
\EE\Big[
\sup_{0\le s\le T}|u(s)|_{H^{-1}_2}^2
+\int_0^T |u(s)|^{\gamma+1}_{L^{\gamma+1}}\, ds\Big]<\infty,
\qquad
\EE \Big[|v|_{L^2(0,T;H^1_2)}^2\Big]<\infty,
\end{align}
proving \eqref{esti reg v1}, \eqref{esti reg v2} and \eqref{esti uv}. This completes the proof of Proposition \ref{prop existence}.
\end{proof}


\begin{proposition}\label{prop conti}
For $\gamma>3$ there exist numbers  $R_1>0$ and $R_2>0$ such that $\CT$ maps $\SubX$ into itself.
\end{proposition}

\begin{proof}
Let $\eta\in \CX_\MA(R_1,R_2)$.
The proposition is shown in three steps. 
 First, in the Step (i), we investigate the norm of
$u$ depending on $v$ and $\eta$. In this way, we get an estimate on $\|u\|_{\mathbb{X}}$ in terms of $R_2$, from which we
get a condition for $R_2$.
Now, to get an estimate on
$\EE\Big[ \sup_{0\le s\le T}|u(s)|_{L^{\gamma+1}}^{\gamma+1}\Big]$, we first establish in Step (ii) an estimate on $v$ in a stronger norm, depending on $R_2$. From this estimate we get in Step (iii) the estimate  $\EE\Big[\sup_{0\le s\le T}|u(s)|_{L^{\gamma+1}}^{\gamma+1}\Big]$ depending on $R_2$, which gives us a condition for $R_1$.
%
%
%
To start, let  $\eta\in \SubX$ and $(u,v)$ be solution to \eqref{sysu}--\eqref{sysv}.
%
%
\del{Observe, first by Example 3.2-(4) \cite{vanNeerven1} we know that the Laplace operator with Neuman boundary condition has a bounded $H^\infty$-calculus on $L^{\gamma+1}(\CO)$. In addition, let us remind that the embedding of $H_2$ in $L^2(\CO)$ is a Hilbert-Schmidt embedding.
Hence, we get by Theorem 4.5-(iii) with $e=v_1-v_2$, $B(e)[\psi]:=e\,\psi$, $\psi\in H_2$, $F(e)=\alpha e$, $e_0=0$ (and incorporating $f(t):=\eta_1(t)-\eta_2(t)$)
\DEQSZ\label{aim1}  \EE \|v_1-v_2\|^{\gamma+1}_{H^2_{\gamma+1}}&\le &
C\, \EE \|\eta_1-\eta_2\|^{\gamma+1}_{L^{\gamma+1}}.
\EEQSZ}
\del{Applying the  It\^o-formula to  the function  $\Phi(v):=|v|_{L^2}^2$, we obtain
\begin{align*}
&{} \frac 12|  v (t)|^2_{L^2}-\frac 12| v_0|^2_{L^2}+r_v \int_0^{t} |\nabla v (s)|_{L^2}^2 \, ds +\alpha \int_0^{t } | v (s)|^2_{L^2}\, ds
=  \beta  \int_0^{t} \la  v (s), \eta(s)\ra \, ds
\notag\\&
+\sum_{k\in\BI}\int_0 ^{t } \la   v (s), v (s)\, \psi^{(\delta_2)}_k\ra d\beta^{(2)}_k(s)
+\frac 12 \sum_{k\in\BI } \int_0 ^{t } | v (s)\,\psi^{(2)}_k|_{L^2}^2\, ds.
\end{align*}
Observe, that there exists a constant $C_m>0$ such that we have for $m=\frac {\gamma+1}{\gamma-1}$
\begin{align*}
&{} \frac 12|  v (t)|^{2m}_{L^2}+r^m_v \Big[\int_0^{t} |\nabla v (s)|_{L^2}^2 \, ds \Big]^m +\alpha^m\Big[ \int_0^{t } | v (s)|^2_{L^2}\, ds\Big]^m
\\
\le &\,\,C_m \Bigg\{ \frac 12| v_0|^{2m}_{L^2}  +\beta^m\Big[   \int_0^{t} \la  v (s), \eta(s)\ra \, ds\Big]^m
\notag\\&
+\Big|\sum_{k\in\BI}\int_0 ^{t } \la   v (s), v (s)\, \psi^{(\delta_2)}_k\ra d\beta^{(2)}_k(s)\Big|^m
+\Big|\frac 12 \sum_{k\in\BI } \int_0 ^{t } | v (s)\,\psi^{(2)}_k|_{L^2}^2\, ds\Big|^m\Bigg\}.
\end{align*}
Taking the supremum over $[0,T]$, then, expectation, the Burkholder-Davis-Gundy inequality, and finally applying the Young inequality, we obtain
\begin{align*}
&\frac 12 \EE\Big[\sup_{0 \leq s \leq T}|  v (s)|^{2m}_{L^2}\Big]
+r_v  \EE\Big[\int_0^{T} |\nabla v (s)|_{L^2}^2 \, ds\Big]^m +\frac{\alpha}{2}  \EE\Big[\int_0^{T} | v (s)|^2_{L^2}\, ds\Big]^m
\notag\\&
\leq C_m\Bigg\{ \frac 12 \EE| v_0|^{2m}_{L^2}+ \frac{ \beta^2}{2 \alpha}  \EE\Big[ \int_0^{T} |\eta(s)|^2 _{L^2}\, ds\Big]^m
+
  \frac{\ep}{4} \,\EE\Big[\sup_{ 0\le s \le T} | v(s)|_{L^2}^2\Big]^m
+ C(\ep) \EE\Big[\int_0^{T } | v(s)|_{L^2}^2 ds\Big]^m\Bigg\}
.
\end{align*}
Setting $\ep=1/C_m$ and applying the Gronwall Lemma, we get
\begin{align*}
&\frac 12 \EE\Big[\sup_{0 \leq s \leq T}| v (s)|^{2m}_{L^2}\Big]
+r_v  \EE\Big[\int_0^{T} |\nabla v (s)|_{L^2}^2 \, ds \Big]^m +\frac{\alpha}{2}  \EE\Big[\int_0^{T} | v (s)|^2_{L^2}\, ds\Big]^m
\notag\\&
\leq (1+C T e^{CT})
\Big[\frac 12 \EE| v_0|^{2m}_{L^2} + \frac{ \beta^2}{2 \alpha}\EE\Big(\int_0^T |\eta(s)|^2_{L^2}ds\Big)^m\Big]
\notag\\
&\le
(1+C T e^{CT})
\Big[\frac 12 \EE| v_0|^{2\frac {\gamma+1}{\gamma-1}}_{L^2} + \frac{ \beta^2}{2 \alpha}
\Big\{ \EE \int_0^T |\eta(s)|_{L^{\gamma+1}}^{\gamma+1} \, ds \,\Big\}^{\frac {2} {\gamma-1}}\Big].
\end{align*}
Since $\eta\in \CX_\mathfrak{A}(R_1,R_2)$ we have shown that there exists a constant $C>0$ such that we have
\begin{align*}\label{estimatenablav}
&\frac 12 \EE[\sup_{0 \leq s \leq T}|  v (s)|^{2m}_{L^2}]
+r_v  \EE[\int_0^{T} |\nabla v (s)|_{L^2}^2 \, ds]^m +\frac{\alpha}{2}  \EE[\int_0^{T} | v (s)|^2_{L^2}\, ds]^m
\notag\\&
\le
(1+C T e^{CT})
\Big[\frac 12 \EE| v_0|^{2m}_{L^2} + \frac{ \beta^2}{2 \alpha}
R_2^{\frac 2 {\gamma-1}}\Big].
\end{align*}}
\begin{steps2}
\item
Applying the It\^o formula to the function $\Phi(u):=|u|_{H^{-1}_2}^2$ we get
\begin{align*}
 |u(t)|_{H^{-1}_2}^2-|u_0|_{H^{-1}_2}^2
&=\int_0^ t \Big[ \la u(s),r_u \Delta u^{[\gamma]}(s)\ra_{H^{-1}_2}
-\chi\,
\la u(s), \di( \eta(s)\nabla v(s)\ra_{H^{-1}_2}\, ds
\notag\\
& \quad
 +\sum_{k\in\BI_1}  \int_0^t \la u(s),\sigma_u u(s)\psi^{(1)}_k \ra_{H^{-1}_2}\, d\beta_k^{(1)}(s)
+\frac 12 \sum_{k\in\BI_1} \int_0^t\sigma_u^2 |u(s)\psi^{(1)}_k|^2_{H^{-1}_2}\, ds.
\end{align*}
Taking supremum over $[0,T]$ and then expectation, it is easy to see that
\begin{align*}
&\frac{ 1}{2} \,\EE\Big[\sup_{0 \leq s \leq T} |u(s)|_{H^{-1}_2}^2\Big]
-\frac 12 \EE\,|u_0|_{H^{-1}_2}^2
\\
&
\le  \EE\Big[ \sup_{0 \leq s \leq t} \sum_{k\in\BI_1} \Big| \int_0^s \la u(r),\sigma_u u(r)\psi^{(1)}_k \ra_{H^{-1}_2}\, d\beta_k^{(1)}(r)\Big|\Big]
 +C\,\dfrac{\sigma_u^2}{2} \,\EE\Big[\int_0^T |u(s)|^2_{H^{-1}_2} ds\Big]
\notag\\
& \quad -\EE \Big[\sup_{0 \leq s \leq t} \int_0^s \int_{\CO} \big[r_u u^{[\gamma]}(r,x)+\chi (-\nabla)^{-1}(\eta(r,x)\nabla v(r,x))u(r,x)\big]dx\,dr \Big]
.
\end{align*}
The Burkholder-Davis-Gundy inequality (see \eqref{BGI2}) and the Young inequality for product term gives for a constant $C>0$
\begin{align}
& \EE\Big[ \sup_{0 \leq s \leq t}\sum_{k\in\BI_1}\Big| \int_0^s \la u(r),\sigma_u u(r)\psi^{(1)}_k \ra_{H^{-1}_2}\, d\beta_k^{(1)}(r)\Big|\Big]
\leq C \EE\Big[ \Big(\int_{0}^{t} \sigma_u^2 | u(s)|^2_{H^{-1}_2} ds \Big)^{\frac 12}\Big]
\notag\\
& \le \frac 14 \EE\Big[\sup_{0 \leq s \leq T} |u(s)|_{H^{-1}_2}^2\Big]
+C\, \sigma_u^2  \EE\Big[ \int_{0}^{T} |u(s)|_{H^{-1}_2}^2\, ds\Big].
\end{align}
To estimate the nonlinear term we apply first the generalized H\"older inequality with $p=\gamma+1$ and
its conjugate $p'=\frac {\gamma+1}\gamma$.  Then we apply the Sobolev embedding $L^1(\CO)\hookrightarrow H^{-1}_{\frac {\gamma+1}\gamma}(\CO)$, and, finally again the H\"older inequality. In this way we get
\DEQS
\Big|\int_\CO (-\nabla)^{-1}(\eta(s,x)\nabla v(s,x))\, u(s,x)dx\Big| \le |u(s)|_{L^{\gamma+1}}  | (-\nabla)^{-1}(\eta(s)\nabla v(s))|_{L^{\frac {\gamma+1}\gamma}}
\\
\le  |u(s)|_{L^{\gamma+1}}  | \eta(s)\nabla v(s)|_{L^1}
\le  |u(s)|_{L^{\gamma+1}}  | \eta(s)|_{L^{\gamma+1}}  |\nabla v(s)|_{L^{\frac {\gamma+1}\gamma}} .
\EEQS
Substituting this estimate above we get for $\ep>0$
\begin{align*}
 &\frac 14\EE\Big[ \sup_{0\le s\le t}|u(s)|_{H^{-1}_2}^2\Big] -\frac 12\EE |u_0|_{H^{-1}_2}^2
\le -r_u\EE \Big[\int_0^T |u(s)|_{L^{\gamma+1}}^{\gamma+1}\,ds\Big]
+C\EE \Big[ \int_0^T|u(s)|_{H^{-1}_2}^{2}\, ds \Big]
\notag\\
& \quad
+\chi\,C\,\EE\Big[ \int_0^T \Big\{|\nabla v(s)|_{L^{\frac{\gamma+1}\gamma}}|\eta(s)|_{L^{\gamma+1}}|u(s)|_{L^{\gamma+1}}\Big\}\, ds\Big]
\notag\\
&\le -r_u\EE\Big[\int_0^T|u(s)|_{L^{\gamma+1}}^{\gamma+1}\, ds \Big]
+\ep \EE\Big[ \int_0^ T  |\eta(s)|_{L^{\gamma+1}}^ {{\gamma+1}}\, ds \Big]
 +  C(\ep,r_u,\gamma)\EE\Big[\int_0^ T|\nabla v(s)|_{L^{\frac{\gamma+1}\gamma}}^{\frac {\gamma+1}{\gamma-1}}\, ds\Big]
\notag\\
& \quad
 +\frac {r_u}{2}
\EE\Big[ \int_0^ T |u(s)|_{L^{\gamma+1}}^{\gamma+1}\, ds\Big]
+C\EE\Big[\int_0^ T |u(s)|_{H^{-1}_2}^{2}\, ds\Big]
.
\end{align*}
Rearranging gives
\DEQS
\lqq{\frac 14 \EE\Big[ \sup_{0\le s\le t}|u(s)|_{H^{-1}_2}^2\Big] -\frac 12 \EE |u_0|_{H^{-1}_2}^2
+\frac {r_u}{2}\EE \Big[
\int_0^ T |u(s)|_{L^{\gamma+1}}^{\gamma+1}\, ds\Big]} &&
\\
&\le& \ep \EE\Big[ \int_0^ T  |\eta(s)|_{L^{\gamma+1}}^ {{\gamma+1}}\, ds\Big] +  C(\ep,r_u,\gamma)
\EE \Big[\int_0^ T|\nabla v(s)|_{L^{\frac{\gamma+1}\gamma}}^{(\gamma+1)/(\gamma-1)}\, ds\Big]
 +C\EE\Big[\int_0^ T |u(s)|_{H^{-1}_2}^2\, ds\Big]
.
\EEQS
To estimate the second term on the right hand side, we take into account that
for $\gamma>3$ we know ${(\gamma+1)/(\gamma-1)}\le 2$.
In addition from  Claim \ref{existencev}-(a) we know that there exists a constant $C>0$ such that
\DEQSZ\label{ueberleg1}
\EE \|v\|_{L^2(0,T;H^1_2)}^2\le \EE |v_0|_{L^2}^2+C\EE \|\eta\|^2_{L^2(0,T;H^{-1}_2)}.
\EEQSZ
In this way we obtain
\begin{align}\label{ueberleg2}
 \EE \|v\|_{L^\frac {\gamma+1}{\gamma-1}(0,T;H^1_{\frac {\gamma+1}\gamma})} ^\frac {\gamma+1}{\gamma-1}
\le C\lk( \EE \|v\|_{L^{2}(0,T;H_{2}^1)}^{2}\rk)^\frac {\gamma+1}{2(\gamma-1)}
\le  \,C\lk(\EE|v_0|_{L^2} 
^{2}+\EE \|\eta\|^2_{L^{2}(0,T;H^{-1}_2)}\rk)^\frac {\gamma+1}{2(\gamma-1)}.
\end{align}
\del{
$\gamma\ge2$, and therefore $\frac {\gamma+1}{\gamma-1}\le \gamma+1$.
Using Remark \ref{remark11} and Claim \ref{existencev} we obtain
\DEQS
\lqq{  \EE \|v\|_{L^\frac {\gamma+1}{\gamma-1}(0,T;H^1_{\frac {\gamma+1}\gamma})} ^\frac {\gamma+1}{\gamma-1}
\le C\lk( \EE \|v\|_{L^{\gamma+1}(0,T;H_{\gamma+1}^1)}^{\gamma+1}\rk)^\frac {\gamma+1}{\gamma-1} }&&
\\
&\le &  \,C\lk(\EE|v_0|_{H^{\frac {2 \gamma}{\gamma+1}}} 
^{\gamma+1}+\EE \|\eta\|_{L^{\gamma+1}(0,T;L^{\gamma+1})}^{\gamma+1}\rk)^\frac {1}{\gamma-1}
.\EEQS}
Since $\eta\in \CX_\MA(R_1,R_2)$, we can write
\DEQS
 \EE \|v\|_{L^\frac {\gamma+1}{\gamma-1}(0,T;H^1_{\frac {\gamma+1}\gamma})} ^\frac {\gamma+1}{\gamma-1}
\le  C\,
\Big[ \EE| v_0|^{2}_{L^2} +
R_2 \Big]^\frac {\gamma+1}{2(\gamma-1)}
.
\EEQS
Using the Gronwall Lemma we know that for any $\ep>0$ there exist  constants $C_1,C_2>0$ such that
\DEQSZ\label{ueberleg3}
\lqq{\EE\Big[\sup_{0 \leq t \leq T} |u(t)|_{H^{-1}_2}^2\Big]
+\frac {r_u}{4}\EE \Big[\int_0^ T |u(s)|_{L^{\gamma+1}}^{\gamma+1}\, ds\Big]}
&&\notag\\
&\le&
\frac 12 \EE |u_0|_{H^{-1}_2}^2+ \ep R_2+ C_1\,R_2^\frac {\gamma+1}{2(\gamma-1)} +C_2\,\Big( \EE| v_0|^{2}_{L^2}\Big)^\frac {\gamma+1}{2(\gamma-1)}
\EEQSZ
\del{Estimate \eqref{estimatenablav} shows that for any $\ep>0$, there exists a constants $C_1,C_2>0$ such that we have
\DEQS
\lqq{ \EE\Big[\sup_{0 \leq t \leq T} |u(t)|_{H^{-1}_2}^2\Big]
+\frac {r_u}{2}\EE\Big[\int_0^ T |u(s)|_{L^{\gamma+1}}^{\gamma+1}\, ds \Big]}
&&
\\
&\le&
\frac C2 \EE |u_0|_{H^{-1}_2}^2+C\, \ep R_2+ CR_2^{\frac 1{\gamma-1}} +\Big(\EE| v_0|^{2}_{L^2}\Big)^\frac {\gamma+1}{2(\gamma-1)} .
\EEQS}
Taking $\ep=\frac{r_u}{8}$ and $R_2$ so large that
\DEQS
 R_2\ge \frac 8{r_u}\Big\{ C_1 R_2^\frac {\gamma+1}{2(\gamma-1)} + \frac 12 \EE |u_0|_{H^{-1}_2}^2+
C_2\, \Big( \EE| v_0|^{2}_{L^2}\Big)^\frac {\gamma+1}{2(\gamma-1)}\Big\}.
\EEQS
Consequently,
$$
\EE\Big[\sup_{0 \leq t \leq T} |\eta(t)|_{H^{-1}_2}^2 \Big]
\le R_2,
$$
which essentially implies
$$
 \EE \Big[\sup_{0 \leq t \leq T} |u(t)|_{H^{-1}_2}^2\Big]
+\frac {r_u}{4}\EE \Big[\int_0^ T |u(s)|_{L^{\gamma+1}}^{\gamma+1}\, ds\Big] \le R_2.
$$

\medskip

\item
Next, we derive a lower estimate for $R_1$.
Applying the It\^o formula to the function
$\Phi(u):=|u|_{L^{\gamma+1}}^{\gamma+1}$ and using standard calculation we obtain
\DEQS 
\lqq{ d |u(t)|_{L^{\gamma+1}}^{\gamma+1}
}
\\
&= &{(\gamma+1)}\Big(\int_\CO u^{{\gamma}}(s,x) \Delta u^\gamma (t,x) \, dx \Big)dt
- {(\gamma+1)} \Big(\int_\CO u^{\gamma}(t,x) \di(\eta (t,x) \nabla v(t,x))\,dx\Big) dt
\notag\\
&&
+\dfrac{(\gamma+1)\gamma}{2}\sigma^2_u \sum_{k\in\BI}
\int_\CO  u^{{\gamma+1}}(t,x)\psi^{(1)}_k(x)\psi^{(1)}_k(x) \, dx\, dt
\\
&&{}+ ({\gamma+1}) \sigma_u \sum_{k\in\BI} \int_\CO u^{\gamma+1}(t,x) \psi^{(1)}_k(x)\,d\beta_k^{(1)}(t)
\notag\\
&\le &{}-({\gamma+1})\int_\CO \nabla u^{\gamma } (t,x)\nabla u^{\gamma}(t,x) dx\, dt
+ ({\gamma+1})\gamma \int_\CO u^{\gamma-1} (t,x)\nabla u(t,x) \nabla v(t,x) \eta (t,x) dx\, dt
\notag\\
& &{} +(\gamma+1) \gamma \sigma^2_u  \Gamma_{\infty}^{(1)}  |u(t)|_{L^{\gamma+1}}^{\gamma+1}
 + ({\gamma+1}) \sigma_u \sum_{k\in\BI} \int_\CO u^{\gamma+1}(t,x) \psi^{(1)}_k(x)\,d\beta_k^{(1)}(t).
\EEQS 
This implies
\begin{align*}
&d |u(t)|_{L^{\gamma+1}}^{\gamma+1}+({\gamma+1})\gamma^2\Big( \int_{\CO}u^{2\gamma-2}|\nabla u(t,x)|^2 dx\Big) dt
\notag\\
&=({\gamma+1})\gamma\Big( \int_{\CO} u^{\gamma-1}(t,x) \eta(t,x) \nabla u (t,x)\cdot \nabla v(t,x)\, dx \Big) dt
+ C(\Gamma_\infty^{(1)}, \sigma_u, \gamma) |u(t)|_{L^{\gamma+1}}^{\gamma+1}dt\notag
\\
&{} +({\gamma+1}) \sigma_u \sum_{k\in\NN} \int_\CO u^{\gamma+1}(t,x) \psi^{(1)}_k(x)\,d\beta_k^{(1)}(t).
\end{align*}
Integrating over  $[0,T]$ we obtain
\begin{align}\label{esti ulp}
&|u(t)|_{L^{\gamma+1}}^{\gamma+1}-|u_0|_{L^{\gamma+1}}^{\gamma+1} +({\gamma+1})\gamma^2 \int_0^t \int_{\CO}u^{2\gamma-2}|\nabla u(t,x)|^2 dx\,ds
\notag\\
&=({\gamma+1})\gamma\int_0^t \int_{\CO} u^{\gamma-1}(s,x) \eta(s,x) \nabla u(s,x) \cdot \nabla v(s,x)\, dx\, ds
+ C(\Gamma_\infty^{(1)}, \sigma_u, \gamma ) \int_0^t |u(s)|_{L^{\gamma+1}}^{\gamma+1} ds
\notag\\
& \quad
 +({\gamma+1})\sigma_u \sum_{k\in\NN} \int_\CO u^{\gamma+1}(t,x) \psi^{(1)}_k(x)\,d\beta_k^{(1)}(t).
\end{align}
Using the H\"older inequality with $p=\frac{\gamma+1}{2}$ and $p'=\frac{\gamma+1}{\gamma-1}$,
we can write
\begin{align}\label{esti nonlin}
& ({\gamma+1})\gamma\int_{\CO} u^{\gamma-1}(s,x) \eta(s,x) \nabla u(s,x) \cdot \nabla v(s,x)\, dx\, ds
\notag\\
&\leq \dfrac{({\gamma+1})\gamma^2}{2}|\nabla u^{\gamma}(s)|_{L^2}^2  +C_\gamma  \int_{\CO} \eta^2(s,x) |\nabla v(s,x)|^2 dx
 \notag\\
 & \leq \dfrac{({\gamma+1})\gamma^2}{2} | u^{\gamma-1}(s) \nabla u(s)|_{L^2}^2 +C_\gamma | \eta^{\gamma+1}(s)|_{L^\gamma+1}^{{2}}\,
|\nabla v(s)|_{L^{\frac{(\gamma+1)}{\gamma-1}}}^{2}
 \notag\\
  & \leq \dfrac{({\gamma+1})\gamma^2}{2}\int_{\CO} u^{2(\gamma-1)} (s)|\nabla u(s)|^2 dx +\eps_1 |\eta(s)|_{L^{\gamma+1}}^{\gamma+1}
  +C(\eps_1,\gamma) |\nabla v(s)|_{L^{\frac{2(\gamma+1)}{\gamma-1}}}^{\frac{2(\gamma+1)}{\gamma-1}}.
\end{align}
Using \eqref{esti nonlin}, we obtain from \eqref{esti ulp} that
\begin{align}\label{esti ulp 1}
&|u(t)|_{L^{\gamma+1}}^{\gamma+1}-|u_0|_{L^{\gamma+1}}^{\gamma+1} +({\gamma+1})\gamma^2 \int_0^t \int_{\CO}u^{2\gamma-2}|\nabla u(t,x)|^2 dx\,ds
\notag\\\nonumber
&\le ({\gamma+1})\gamma\int_0^t \int_{\CO} u^{\gamma-1}(s,x) \eta(s,x) \nabla u(s,x) \cdot \nabla v(s,x)\, dx\, ds
\\\nonumber
&\quad +C(\eps_1,\gamma) \int_0^t C(\eps_1,\gamma) |\nabla v(s)|_{L^{\frac{2(\gamma+1)}{\gamma-1}}}^{\frac{2(\gamma+1)}{\gamma-1}}
+\eps_1 \int_0^t \int_{\CO} |\eta(s,x)|^{\gamma+1}dx \,ds
\notag\\
& \quad
+C_\gamma \int_0^t |u(s)|_{L^{\gamma+1}}^{\gamma+1} ds
 +({\gamma+1})\sigma_u \sum_{k\in\NN} \int_\CO u^{\gamma+1}(t,x) \psi^{(1)}_k(x)\,d\beta_k^{(1)}(t).
\end{align}
Note, that
\DEQS
\int_0^ T |\nabla v(s)|_{L^{\frac{2(\gamma+1)}{\gamma-1}}}^{\frac{2(\gamma+1)}{\gamma-1}}\, ds 
& \le & C\lk( \int_0^ T|v(s)|_{H^1_{4}}^{4}\,ds\rk)^\frac{\gamma+1}{2( \gamma-1)}.
\EEQS
Claim \ref{existencev} yields 
\DEQSZ\label{zww}
\EE \Big[\int_0^ T |\nabla v(s)|_{L^{\frac{2(\gamma+1)}{\gamma-1}}}^{\frac{2(\gamma+1)}{\gamma-1}}\, ds \Big]&\le& \lk\{ \EE|v_0|^{4}_{L^4}+
\EE\Big[\int_0^ T|\eta(s)|_{H^{-1}_4}^4\,ds\Big]\rk\}^\frac{\gamma+1}{2( \gamma-1)}.
\EEQSZ
Taking supremum over $[0,T]$ and then expectation, using the Burkholder--Davis--Gundy inequality, i.e.\ \eqref{BGIgamma},  and substituting \eqref{zww} in \eqref{esti ulp 1}, we get
\begin{align}\label{exercise}
&\EE\Big[\sup_{0 \le s \le T}|u(s)|_{L^{\gamma+1}}^{\gamma+1}\Big]-\EE|u_0|_{L^{\gamma+1}}^{\gamma+1}
+ \dfrac{({\gamma+1})\gamma^2}{2} \EE\Big[ \int_0^T \int_{\CO} (u(s,x))^{2(\gamma-1)} |\nabla u(s,x)|^2 dx\,ds\Big]
\notag\\ &
 \leq C(\ep_1,\gamma)\EE\Big[ \int_0^ T |\nabla v(s)|_{L^{\frac{2(\gamma+1)}{\gamma-1}}}^{\frac{2(\gamma+1)}{\gamma-1}}\, ds\Big]
 +\eps_1 \EE \Big[\int_0^T |\eta(s)|^{\gamma+1}_{L^{\gamma+1}}ds \Big]
+C\EE \Big[ \int_0^T |u(s)|_{L^{\gamma+1}}^{\gamma+1} ds \Big]
\notag
\\
&\quad+\ep\EE \Big[ \int_0^T |u(s)|_{H^\delta_{\gamma+1}}^{\gamma+1} ds\Big]
\notag \\
& \leq  C(\ep_1,\gamma) \lk\{ \EE|v_0|^{4}_{L^4}+\notag
\EE \Big[\int_0^ T|\eta(s)|_{H^{-1}_4}^4\,ds\Big] \rk\}^\frac{\gamma+1}{2( \gamma-1)}
+ \eps_1 \EE \Big[\int_0^T |\eta(s)|_{L^{\gamma+1}}^{\gamma+1} ds \Big]
\\
&\quad +C\EE \Big[ \int_0^T |u(s)|_{L^{\gamma+1}}^{\gamma+1} ds\Big]+\ep\EE \Big[ \int_0^T |u(s)|_{H^\delta_{\gamma+1}}^{\gamma+1} ds\Big]
\notag\\
& \leq  C(\ep_1,\gamma) \lk( \EE|v_0|^{4}_{L^4}\rk) ^\frac{\gamma+1}{2( \gamma-1)}+
\eps_1 R_2 +C(\eps_1,{\gamma}) R_2 ^\frac{\gamma+1}{2( \gamma-1)}
+ C({\gamma},\eps_1)\EE \Big[ \int_0^T |u(s)|_{L^{\gamma+1}}^{\gamma+1} ds\Big]
\notag \\
&\quad +\ep\EE \Big[ \int_0^T |u(s)|_{H^\delta_{\gamma+1}}^{\gamma+1} ds\Big].\
\end{align}
 Note, that by the Technical Proposition \ref{runst1} we cancel  the term $\ep\EE \Big[ \int_0^T |u(s)|_{H^\delta_{\gamma+1}}^{\gamma+1} ds\Big]$
 by $$\EE\Big[ \int_0^T \int_{\CO} (u(s,x))^{2(\gamma-1)} |\nabla u(s,x)|^2 dx\,ds\Big].$$
Using the Gronwall inequality we get
\begin{align}
&\EE\Big[\sup_{0 \le s \le T}|u(s)|_{L^{\gamma+1}}^{\gamma+1}\Big]
+ \dfrac{({\gamma+1})\gamma^2}{4} \EE\Big[ \int_0^T \int_{\CO} (u(s,x))^{2(\gamma-1)} |\nabla u(s,x)|^2 dx\,ds\Big]
\notag\\ &\leq \Big( \EE|u_0|_{L^{\gamma+1}}^{\gamma+1}+ C(\ep_1,\gamma) \lk( \EE|v_0|^{4}_{L^4}\rk) ^\frac{\gamma+1}{2( \gamma-1)}
+ \eps_1 R_2 +C(\eps_1,{\gamma}) R_2 ^\frac{\gamma+1}{2( \gamma-1)}\Big)(1+CT)e^{CT}
\notag\\
& \leq C \Big( \EE|u_0|_{L^{\gamma+1}}^{\gamma+1}+ C(\ep_1,\gamma) \lk( \EE|v_0|^{4}_{L^4}\rk) ^\frac{\gamma+1}{2( \gamma-1)}
+ \eps_1 R_2 +C(\eps_1,{\gamma}) R_2 ^\frac{\gamma+1}{2( \gamma-1)}\Big).
\end{align}
Choosing
$$\ep_1C\, R_2
+C(\eps_1,{\gamma}) R_2 ^\frac{\gamma+1}{2( \gamma-1)} +C\,\EE |u_0|_{L^{\gamma+1}}^{\gamma+1}+ C(\ep_1,\gamma) \lk( \EE|v_0|^{4}_{L^4}\rk) ^\frac{\gamma+1}{2( \gamma-1)} \le R_1,
$$
we know by the calculations above that
\begin{align}
&\EE\Big[\sup_{0 \le s \le T}|u(s)|_{L^{\gamma+1}}^{\gamma+1}\Big]
+ \dfrac{{(\gamma+1)}\gamma^2}{2} \EE\Big[\sup_{0 \le t \le T} \int_0^t\int_{\CO} (u(s,x))^{2(\gamma-1)} |\nabla u(s,x)|^2 dx\,ds\Big]
\le R_1.
\end{align}

\medskip

Summarising,
 we have shown that there exists $R_1>0$ and $R_2>0$ such that $\CT$ maps $\SubX$ in itself, which finishes the proof of Proposition \ref{prop conti}.
 \end{steps2}
\end{proof}

In the next Proposition we show the continuity of the solution operator.

\begin{proposition}\label{continuity}
There exists some $\delta>0$ and a constant $C>0$ such that for all $\eta_1,\eta_2\in\mathcal{X}_\MA(R_1,R_2)$ we have
\DEQS
\|\CT[\eta_1]-\CT[\eta_2]\|_\BX\le C \|\eta_1-\eta_2\|_{\BX}^\delta.
\EEQS
\end{proposition}

\begin{proof}[Proof of Proposition \ref{continuity}]
Let $(u_1,v_1)$ and $(u_2,v_2)$ be solutions of \eqref{sysu} and \eqref{sysv} corresponding to $\eta_1$ and $\eta_2$ respectively, i.e.,
\DEQSZ
d{u}_i(t)&=& \Big( r_u \Delta u_i^\gamma (t) 
-\chi \mbox{div} \big( \eta_i(t)\nabla v_i(t)\big)\Big)\, dt+\sigma_u u_i(t)dW_1(t)\label{sysu12}
\\
d{v}_i(t) &=& \big(r_v \Delta v_i(t)+\beta u_i(t) 
-\alpha v_i(t)\Big)\, dt +\sigma_v v_i(t) dW_2(t)\,\quad t\in [0,T],\phantom{\big|}\,\,i=1,2.\label{sysv12}
\EEQSZ

\medskip
First, let us investigate the difference of $e=v_1-v_2$. Since the system of $v$ is linear, we know that $e$ solves
\DEQS
d{e}(t) &=& \big(r_v \Delta e(t)+\beta e(t) 
-\alpha e(t)\Big)\, dt +\sigma_v e(t) dW_2(t)\,\quad t\in [0,T],\phantom{\big|}\,\,i=1,2.%
\EEQS
Secondly,  by Example 3.2-(4) \cite{vanNeerven1} we know that the Laplace operator with Neumann boundary condition has a bounded $H^\infty$-calculus on $H^{-1}_2(\CO)$.
Hence, we know by Theorem 4.5-(iii) with $B(e)[\psi]:=e\,\psi$, $\psi\in \mathcal H$, $F(e)=\alpha e$, $e_0=0$ (and incorporating $f(t):=\eta_1(t)-\eta_2(t)$)
that there exists a constant $C>0$ such that
\DEQSZ\label{aim3}  \EE \|v_1-v_2\|^{2}_{L^2(0,T;H^1_{2})}&\le &
C\, \EE\Big[ \sup_{0 \le s \le t}\|\eta_1(s)-\eta_2(s)\|^2_{H^{-1}_2}\Big].
\EEQSZ

\medskip

Next, we investigate the difference of $u_1-u_2$.
Using It\^o formula to $|u_1(t)-u_2(t)|^2_{H^{-1}_2}$ and by canonical calculations we obtain
\begin{align}
 &\frac 12|u_1(t)-u_2(t)|_{H^{-1}_2}^2
\le \Big[\int_0^ t \la (u_1(s)-u_2(s)),r_u \Delta (u_1^{\gamma}(s)-u_2^{\gamma}(s))\ra_{H^{-1}_2}ds \Big]
\notag
\\
&\quad-\chi \Big[\int_0^ t   \la u_1(s)-u_2(s)),\di(\eta_1(s)\nabla v_1(s)-\eta_2(s)\nabla v_2(s))\ra_{H^{-1}_2}ds \Big]
\notag
\\
& \quad\notag
+\sum_{k\in\mathbb{I}_1}\int_0^t \la \big(\sigma_u u_1(s)-\sigma_u u_2(s) \big),( u_1(s)-u_2(s))\psi^{(1)}_k\ra_{H_2^{-1}}\, d\beta^{(1)}_k(s)
\\
\notag&
\quad +\frac12\sum_{k\in\mathbb{I}_1}\int_0^t \sigma_u^2 |(u_1(s)- u_2(s))\psi^{(1)}_k|_{H_2^{-1}}^2\,ds
\\
\notag& \leq - \int_0^t \int_{\CO} \Big[ r_u \big(u_1^{\gamma}(s,x)-u_2^{\gamma}(s,x)\big)
\big(u_1(s,x)-u_2(s,x)\big)dx\,ds
\Big]\notag
\\
\notag
& \quad -\chi \Big[
\int_0^t \int_{\CO} (-\nabla)^{-1}
\Big(\eta_1(s,x)\nabla v_1(s,x)-\eta_2(s,x)\nabla v_2(s,x)
\Big) \big(u_1(s,x)-u_2(s,x)\big)dx\,ds
\Big]\notag
\\\notag
& \quad
+\sum_{k\in\mathbb{I}_1}\int_0^t \la \big(\sigma_u u_1(s)-\sigma_u u_2(s) \big),( u_1(s)-u_2(s))\psi^{(1)}_k\ra_{H_2^{-1}}\, d\beta^{(1)}_k(s)
\\&\notag
 \quad+\frac12\sum_{k\in\mathbb{I}_1}\int_0^t \sigma_u^2 |(u_1(s)- u_2(s))\psi^{(1)}_k|_{H_2^{-1}}^2\,ds
\\&:=J_1(t)+J_2(t)+J_3(t)+J_4(t).
\end{align}
Exploiting the fact that $\big(u_1^{\gamma}-u_2^{\gamma}\big)
(u_1-u_2) \geq |u_1-u_2|^{\gamma+1}$, we obtain
\begin{align}\label{esti u12 gam}
&-r_u \int_0^t \int_{\CO} \Big[ \big(u_1^{\gamma}(s,x)-u_2^{\gamma}(s,x)\big)
\big(u_1(s,x)-u_2(s,x)\big)\Big]dx\,ds
\notag\\
& \leq -r_u \int_0^1 \int_{\CO} |u_1(s,x)-u_2(s,x)|^{\gamma+1}dx\,ds
=-r_u \Big[\int_0^t |u_1(s)-u_2(s)|^{\gamma+1}_{L^{\gamma+1}}ds
\big].
\end{align}
Using \eqref{esti u12 gam} and taking supremum over $[0,T]$ and expectation we obtain
\begin{align*}\label{here_tt}
&\EE \Big[\sup_{0 \leq t \leq T}|u_1(t)-u_2(t)|^2_{{H^{-1}_2}}\Big]
 +r_u \EE \Big[\int_0^T |u_1(s)-u_2(s)|^{\gamma+1}_{L^{\gamma+1}}ds
\Big]
\notag\\\notag
& \leq \chi\EE \Big[\sup_{0 \leq t \leq T} \Big|
\int_0^t \int_{\CO} (-\nabla)^{-1}
\Big(\eta_1(s,x)\nabla v_1(s,x)-\eta_2(s,x)\nabla v_2(s,x)
\Big)
\big(u_1(s,x)-u_2(s,x)\big)dx\,ds
\Big|\Big]\notag\\
& \quad
+\EE \Big[\sup_{0 \leq t \leq T} \Big|\sum_{k\in\mathbb{I}}\int_0^t \la \big(\sigma_u u_1(s)-\sigma_u u_2(s) \big),( u_1(s)-u_2(s))\psi^{(1)}_k\ra_{H_2^{-1}}\, d\beta^{(1)}_k(s)\Big|\Big]
\\&
{}+\frac12\sum_{k\in\mathbb{I}}\int_0^t \sigma_u^2 |(u_1(s)- u_2(s))\psi^{(1)}_k|_{H_2^{-1}}^2\,ds
=J_1(t)+J_2(t)+J_3(t).
\end{align*}
Now we consider the term $J_2(t)$. First, we split the term into the following sum
\begin{align}
&\EE \Big[\sup_{0 \leq t \leq T} \Big|
\int_0^t \int_{\CO} (-\nabla)^{-1}
\Big(\eta_1(s,x)\nabla v_1(s,x)-\eta_2(s,x)\nabla v_2(s,x)
\Big)
\big(u_1(s,x)-u_2(s,x)\big)dx\,ds \Big|\Big]\notag\\
& \leq \EE \Big[\sup_{0 \leq t \leq T}
\underbrace{\Big| \int_0^t \int_{\CO} (-\nabla)^{-1}
\big(\eta_1(s,x)-\eta_2(s,x)\big)\nabla v_1(s,x)
\big(u_1(s,x)-u_2(s,x)\big)dx\,ds\Big|}_{=:I_1(t)} \Big]
\notag\\
&\quad + \EE\Big[\sup_{0 \leq t \leq T} \Big|\underbrace{\Big|
\int_0^t \int_{\CO} (-\nabla)^{-1}\Big[ \eta_2(s,x)
\nabla \big(v_1(s,x)-v_2(s,x)\big)\Big] \big(u_1(s,x)-u_2(s,x)\big)dx\,ds\Big|}_{=:I_2(t)} \Big]
\end{align}
Next, we  estimate $\EE \Big[\sup_{0\le t\le T}I_1(t)\Big]$. Using the H\"older and the Young inequality we know that for any $\ep>0$ there exists a constant $C>0$ with
\begin{align*}
\EE \Big[\sup_{0\le t\le T}I_1(t)\Big]
&\leq \EE \Big[
\int_0^T \big|(\eta_1(s)-\eta_2(s))\nabla v_1(s)\big|_{H_{\frac {\gamma+1}\gamma}^{-1}} |u_1(s)-u_2(s)|_{L^{\gamma+1}}ds
\Big]
\\
&\leq C(\ep)\EE \Big[
\int_0^T \big|(\eta_1(s)-\eta_2(s))\nabla v_1(s)\big|_{H_{\frac {\gamma+1}\gamma}^{-1}}^\frac {\gamma+1}\gamma\, ds\Big] +
\ep\EE \Big[
\int_0^T |u_1(s)-u_2(s)|_{L^{\gamma+1}}^{\gamma+1}ds
\Big]
.
\end{align*}
Applying the Sobolev embedding $L^1(\CO)\hookrightarrow H_{\frac {\gamma+1}\gamma}^{-1}(\CO)$ and the H\"older inequality we get
\begin{align*}
\EE \Big[\sup_{0\le t\le T}I_1(t) \Big]
&\leq C(\ep)\EE \Big[
\int_0^T |\eta_1(s)-\eta_2(s)|^\frac {\gamma+1}\gamma_{L^2}|\nabla v_1(s)|^\frac {\gamma+1}\gamma_{L^2}\, ds \Big]+
\ep\EE \Big[
\int_0^T |u_1(s)-u_2(s)|_{L^{\gamma+1}} ^{\gamma+1}ds
\Big].
\end{align*}
 Next, applying 
 complex interpolation we get for some $\rho\in(0,\frac 1\gamma)$
\begin{align}\label{compl interpolation}
 \Big[ H^{\rho}_2(\CO),H^{-1}_2(\CO)  \Big]_{\theta}=L^2(\CO)  
\end{align}
with
 $  \rho (1-\theta)+\theta(-1)=0$ we know
\begin{align}\label{esti compl}
|\eta_1(s)-\eta_2(s)|_{L^2}\le |\eta_1(s)-\eta_2(s)|_{H^{-1}_2} ^{\theta}|\eta_1(s)-\eta_2(s)|_{H^{ \rho}_2} ^{1-\theta}.
\end{align}
This gives for $\theta=\frac {\rho}{1+\rho}<\frac 1 {1+\gamma}$ and $1-\theta>\frac \gamma{1+\gamma}$
\DEQSZ \label{inter esti}
\lefteqn{
\EE \Big[\sup_{0\le t\le T}|I_1(t)|\Big]}&&
\notag\\
&\le &
 C \EE\Big[
\int_0^T |\eta_1(s)-\eta_2(s)|_{H_2^{-1}}^{\frac {\gamma+1}\gamma\theta} |\eta_1(s)-\eta_2(s)|_{H^{\rho}_2}^{\frac {\gamma+1}\gamma(1-\theta)}
 |\nabla v_1(s)|_{L^2}^\frac {\gamma+1}\gamma\, ds
\notag \\
 &&{} +\ep\EE \Big[
\int_0^T |u_1(s)-u_2(s)|_{L^{\gamma+1}}^{\gamma+1}ds
\Big]
\notag
\\
& \leq&
 \lk\{\EE\Big[\sup_{0\le s\le T}|\eta_1(s)-\eta_2(s)|_{H_2^{-1}}^2 \Big]\rk\}^{\frac {(\gamma+1)\theta}{2\gamma}}
\lk\{ \EE \Big[\sup_{0\le s\le T}| \nabla v_1(s)|_{L^2}^{\gamma+1}\Big] \rk\}^{\frac 1\gamma}
\notag\\&&{}
\times \lk\{ \EE \Big[\int_0^T  \lk(|\eta_1(s)|_{H^{\rho}_2}+|\eta_2(s)|_{H^{\rho}_2}\rk)^{2\gamma}
ds\Big]\rk\}^{\frac {(\gamma+1)(1-\theta)}{2\gamma^2}}
+\ep\,\EE \Big[
\int_0^T |u_1(s)-u_2(s)|_{L^{\gamma+1}}^{\gamma+1}ds
\Big]
.
\EEQSZ
Note, choosing $\eps$ small enough so that the first term in the right hand side of \eqref{inter esti} can be cancelled using the second term in the left hand side of \eqref{here_tt}. Claim \ref{existencev} gives
$$\EE\Big[ \sup_{0\le s\le t}| \nabla v_1(s)|_{L^2}^{\gamma+1}\Big]
\le C\lk( \EE|v_0|_{H^{\frac {2\gamma}{\gamma+1}}}^{\gamma+1}+ R_2\rk) .
$$
Owing to Technical Proposition \ref{runst1} one achieve
$$
 \EE\Big[ \int_0^t  \lk(|\eta_1(s)|_{H^{\rho}_2}+|\eta_2(s)|_{H^{\rho}_2}\rk)^{2\gamma}
ds\Big]  \le R_1.
$$
This finally gives
\begin{align}\label{inter esti I1}
\EE \Big[\sup_{0\le t\le T}I_1(t)\Big]
&\leq C\lk( 1+\EE|v_0|_{H^{\frac {2\gamma}{\gamma+1}}}^{\gamma+1}+ R_2\rk)^{\frac 1 \gamma}R_1^{\frac{(\gamma+1)(1-\theta)}{2 \gamma^2}} \lk\{\EE\Big[\sup_{0\le s\le T}|\eta_1(s)-\eta_2(s)|_{H_2^{-1}}^2 \Big]\rk\}^{\frac {(\gamma+1)\theta}{2\gamma}}.
\end{align}
Next, we estimate $\EE \Big[\sup_{0\le t\le T}I_2(t)\Big]$.
Again, using the H\"older and the Young inequality and applying the Sobolev embedding $L^1(\CO)\hookrightarrow H_{\frac {\gamma+1}\gamma}^{-1}(\CO)$,
 we know that for any $\ep>0$ there exists a constant $C>0$ with
\begin{align}\label{esti I2}
\EE \Big[\sup_{0\le t\le T}I_2(t)\Big]&\leq
C(\ep)\EE \Big[
\int_0^T |\eta_2(s)|^\frac {\gamma+1}\gamma_{L^2} |\nabla( v_1(s)-v_2(s))|^\frac {\gamma+1}\gamma_{L^2}\, ds \Big]
+\ep\EE \Big[
\int_0^T |u_1(s)-u_2(s)|_{L^{\gamma+1}} ^{\gamma+1}ds
\Big]
\notag\\
&\le
C(\ep)\EE\Bigg[ \lk( \sup_{0\le s\le T}  |\eta_2(s)|^\frac {\gamma+1}\gamma_{L^2} \rk)
\lk( \int_0^T|\nabla( v_1(s)-v_2(s))|^\frac {\gamma+1}\gamma_{L^2}\, ds\rk)\Bigg]
\notag\\&\quad +
\ep\EE \Big[
\int_0^T |u_1(s)-u_2(s)|_{L^{\gamma+1}} ^{\gamma+1}ds
\Big]
\notag\\
&\le
C(\ep)\lk\{ \EE \Big[ \sup_{0\le s\le T}  |\eta_2(s)|^ {\gamma+1}_{L^2}\Big] \rk\}^\frac 1 \gamma
\lk\{\EE \Big[ \int_0^T|\nabla( v_1(s)-v_2(s))|^\frac {\gamma+1}{\gamma-1}_{L^2}\, ds\Big] \rk\}^\frac {\gamma-1}{\gamma}
\notag\\ &\quad +
\ep \EE \Big[
\int_0^T |u_1(s)-u_2(s)|_{L^{\gamma+1}} ^{\gamma+1}ds
\Big].
\end{align}
We choose $\eps$ small enough so that the first term in the right hand side of \eqref{esti I2} can be cancelled using the second term in the left hand side of \eqref{here_tt}.
The term $ \EE\Big[  \sup_{0\le s\le t}  |\eta_2(s)|^ {\gamma+1}_{L^2} \Big]$ can be estimated by $R_1$. It remains to estimate
$$
\EE  \|\nabla( v_1(s)-v_2(s))\|_{L^\frac {\gamma+1}{\gamma-1}(0,T;{L^2})}^\frac {\gamma+1}{\gamma-1}.
$$
First, note that for $\gamma\ge 3$ we have
$$
\EE  \|\nabla( v_1(s)-v_2(s))\|_{L^\frac {\gamma+1}{\gamma-1}(0,T;{L^2})}^\frac {\gamma+1}{\gamma-1}
\le \lk\{ \EE  \|\nabla( v_1(s)-v_2(s))\|_{L^2(0,T;{L^2})}^2\rk\}^\frac {\gamma+1}{2(\gamma-1)} .
$$
Using estimate \eqref{aim3} this term can be handled.
It remains to handle $J_3(t)$ and $J_4(t)$.
Using the Burkholder-Davis-Gundy inequality we obtain
\begin{align}
\EE\Big[ \sup_{0 \leq s \leq t}|J_3(s)|\Big]&=\EE\Big[ \sup_{0 \leq t \leq T} \Big|\int_0^t \la \big(\sigma_u u_1(s)-\sigma_u u_2(s) \big),(u_1(s)-u_2(s))\psi_k^{(1)}\ra_{H_2^{-1}}d\beta_k^{(1)}\Big|\Big]
\notag\\&
\leq C \Big(\EE\Big[\int_0^ t \sigma^2_u |u_1(s)- u_2(s)|_{H_2^{-1}}^2 \,ds\Big]\Big)^{\frac 12}
\notag\\
& \leq \frac {1}4 \EE\Big[\sup_{0 \leq t \leq T}|u_1(s)- u_2(s)|_{H_2^{-1}}^2\Big]
+ C \EE\Big[\int_0^T |u_1(s)- u_2(s)|_{H_2^{-1}}^2\, ds \Big].
\end{align}
Finally, evaluating the trace we obtain 
\begin{align}
\EE\Big[ \sup_{0\le s\le t}J_4(s)\Big]
\leq \frac 12\EE\Big[ \int_0^T | u_1(s)- u_2(s)|_{H_2^{-1}}^2\,ds\Big].
\end{align}
Collecting altogether and substituting the estimates above into \eqref{here_tt} and using the Gronwall inequality, one can infer that there exists a constant $C(R_1,R_2)>0$ and a number~$\delta>0$ such that
\begin{align}
\lqq{ \EE \Big[\sup_{0 \leq s \leq T} |u_1(t)-u_2(t)|^2_{H^{-1}_2}
\Big]}
\notag\\&
\leq C(R_1,R_2)\Big( \Big\{ \EE\Big[ \sup_{0 \leq s \leq T} |\eta_1(s)-\eta_2(s)|^2_{H^{-1}_2}\Big]\Big\}^\delta + \EE \Big[\sup_{0 \leq s \leq T} |\eta_1(s)-\eta_2(s)|^2_{H^{-1}_2}\Big]\Big).
\end{align}
This completes the proof of Proposition \ref{continuity}.

\del{
{\bf Should be on another place, on page 21 (see here) }Due to the linearity of the convolution operator
$
\eta\mapsto \mathfrak{F}(\eta)$, where
$$\mathfrak{F}(\eta)(t):= \int_0^t e^{(t-s)(r_v \DeltaA -\alpha I)}\eta(s)ds,
$$
it follows that for $\eta_1,\eta_2\in \CM^{\gamma+1}_\MA(0,T;L^{\gamma+1}(\CO))$, we have
\DEQS
\EE \|v_1-v_2\|_{L^{\gamma+1}(0,T;H^2_{\gamma+1})}^{\gamma+1}\le C\EE \|\eta_1-\eta_2\|_{L^{\gamma+1}(0,T;L^{\gamma+1})}^{\gamma+1}.
\EEQS}
\end{proof}

In the next proposition we will show that $\CT$ maps $\SubX$ to a precompact set.

\begin{proposition}\label{CC}
For any initial condition $(u_0,v_0)$ satisfying Assumption \ref{init}, and all $R_1>0$ and $R_2>0$ we know that
\begin{itemize}
\item[(a)] there exists $r=r(T,\gamma)>0$ such that for any $\eta \in \CX_\MA(R_1,R_2)$, we have
\begin{align}
\sup_{0 \leq t \leq T} \EE
\|\CT\eta(t)\|_{L^{\gamma+1}}^{\gamma+1}
\le C \EE \|\eta\|_{L^{\gamma+1}(0,T;L^{\gamma+1})}^{\gamma+1}.
\end{align}
\item[(b)] there exists a number $\delta=\delta(T,\gamma)>0$ and $C=C(\delta,T,\gamma,R_1,R_2)>0$ such that for any $0<t_1<t_2 \leq T$ and $\eta \in \SubX$ we have
\begin{align}
\EE
\|\CT\eta(t_1)-\CT\eta(t_2)\|_{H_2^{-1}}^2
\le C |t_1-t_2|^{\delta}.
\end{align}
\end{itemize}
\end{proposition}

\begin{proof}
Part \ref{CC}-(a) is clear, due to the definition of $\SubX$.
Let us start with Part \ref{CC}-(b).
 Applying the It\^o formula to the function
$\Phi(u):=|u-u_0|_{H^{-1}_2}$ and using standard calculations
we obtain
\begin{align}
d |u(t)-u_0|_{H^{-1}_2}^2
&= -(\gamma+1) \int_0^ t\la (u(s)-u_0),u(s)^{[\gamma]}-u_0^{[\gamma]}\ra \, ds
+ (\gamma+1) \int_0^ t\la (u(s)-u_0),u_0^{[\gamma]}\ra \, ds
\notag\\
&\quad + (\gamma+1) \int_0^ t \la \nabla^{-1}(u(s)-u_0),\eta(s)\nabla v(s)\ra \, ds
\notag\\
&\quad +\sum_{k\in\BI_1} \int_0^ t \la \nabla^{-1}(u(s)-u_0),\nabla^{-1} (u(s)\psi_k)\ra d\beta_k^{(1)}(t)+\frac 12 \sum_{k\in\BI_1} \int_0^ t |u(s)\psi_k|_{H^{-1}_2}^2 dt
.
\end{align}
This immediately gives
\DEQS
\lqq{d |u(t)-u_0|_{H^{-1}_2}^2+(\gamma+1) \int_0^ t\la (u(s)-u_0),u(s)^{[\gamma]}-u_0^{[\gamma]}\ra \, ds}
&&
\\
&&{}+\sum_{k\in\BI_1} \int_0^ t \la \nabla^{-1}(u(s)-u_0),\nabla^{-1} (u(s)\psi_k)\ra d\beta_k^{(1)}(t)+\frac 12 \sum_{k\in\BI_1} \int_0^ t |u(s)\psi_k|_{H^{-1}_2}^2 dt
\\
&= &
 (\gamma+1) \int_0^ t\la (u(s)-u_0),u_0^{[\gamma]}\ra \, ds
+ (\gamma+1) \int_0^ t \la \nabla^{-1}(u(s)-u_0),\eta(s)\nabla v(s)\ra \, ds
\\
&&{}+\sum_{k\in\BI_1} \int_0^ t \la \nabla^{-1}(u(s)-u_0),\nabla^{-1} (u(s)\psi_k)\ra d\beta_k^{(1)}(t)+\frac 12 \sum_{k\in\BI_1} \int_0^ t |u(s)\psi_k|_{H^{-1}_2}^2 dt
\\
&:=& K_1(t)+ K_2(t)+K_3(t)+K_4(t).
\EEQS
Applying the H\"older inequality gives for the first term
\DEQS
\EE\Big[ \sup_{0\le s\le t}  K_1(s)\Big] &\le &  (\gamma+1)\,\EE\Big[  \Big\{\int_0^ t |u(s)-u_0|^{\gamma+1}_{L^{\gamma+1}}\, ds\Big\}^{\frac 1{\gamma+1}} t^{\frac \gamma{\gamma+1}}
\, |u_0^{[\gamma]}|_{L^\frac {\gamma+1}\gamma}\Big].
\EEQS
The Cauchy Schwarz inequality implies
\DEQS
\EE\Big[ \sup_{0\le s\le t} K_1(s)\Big] &\le &  (\gamma+1)t^{\frac \gamma{\gamma+1}} \lk\{ \EE\Big[ \int_0^ t |u(s)-u_0|^{\gamma+1}_{L^{\gamma+1}}\, ds\Big] \rk\}^ {\frac 1{\gamma+1}}
\, \lk\{ \EE |u_0|_{L^{\gamma+1}}^{\gamma+1} \rk\}^{\frac \gamma{\gamma+1}}.
\EEQS
Applying the Young inequality, we can cancel the $\EE\Big[ \int_0^ t |u(s)-u_0|^{\gamma+1}_{L^{\gamma+1}}\, ds\Big]$ with the second term on the left hand side.
In order to calculate the next term, we apply integration by parts and the H\"older's inequality. In this way we get
\begin{align*}
 \EE\Big[ \sup_{0\le s\le t} K_2(s)\Big]
 & \le
(\gamma+1) \int_0^ t \la \nabla^{-1}(u(s)-u_0),\eta(s)\nabla v(s)\ra \, ds 
\\
&\le
(\gamma+1) \int_0^ t |u(s)-u_0|_{L^{\gamma+1}}|\eta(s)\nabla v(s)|_{H^{-1}_\frac {\gamma+1}\gamma}\, ds.
\end{align*}
By the Young inequality we know that for all $\ep>0$ there exists a constant $C(\ep)>0$ such that
\DEQS
\lqq{ \EE\Big[ \sup_{0\le s\le t} K_2(s)\Big] }&&
\\
&\le&
\ep \EE \int_0^ t |u(s)-u_0|_{L^{\gamma+1}}^{\gamma+1}\, ds +\EE \int_0^ t |\eta(s)\nabla v(s)|^{\frac {\gamma+1}\gamma}_{H^{-1}_\frac {\gamma+1}\gamma}\, ds.
\EEQS
Using the embedding $L^1(\CO)\hookrightarrow H^{-1}_\frac {\gamma+1}\gamma(\CO)$, and applying the H\"older inequality we get
\begin{align*}
\EE\Big[ \sup_{0\le s\le t} K_2(s)\Big] 
&\le 
\ep \EE\Big[ \int_0^ t |u(s)-u_0|_{L^{\gamma+1}}^{\gamma+1}\, ds \Big]
+t\,\EE\Big[ \sup_{0\le s\le t} |\eta(s)|^{\frac {\gamma+1}\gamma}_{L^{\gamma+1}} \sup_{0\le s\le t} |\nabla v(s)|^{\frac {\gamma+1}\gamma}_{L^{\frac {\gamma+1}\gamma}}\Big].
\end{align*}
Applying the H\"older's inequality again we get
\begin{align*}
 \EE\Big[ \sup_{0\le s\le t} K_2(s)\Big]
&\le
\ep \EE \Big[\int_0^ t |u(s)-u_0|_{L^{\gamma+1}}^{\gamma+1}\, ds \Big]
+t\,\lk\{ \EE\Big[ \sup_{0\le s\le t} |\eta(s)|^{(\frac {\gamma+1}\gamma)^2}_{L^{\gamma+1}}\Big] \rk\}^\frac {\gamma+1}\gamma \lk\{\EE\Big[
\sup_{0\le s\le t} |\nabla v(s)|^{{\gamma+1}}_{L^{\frac {\gamma+1}\gamma}}\Big]\rk\}^\frac 1 \gamma.
\end{align*}
Using Claim \ref{existencev} and
$$
|\nabla v|_{L^{\frac {\gamma+1}\gamma}} \le C| v|_{H^1_{\frac {\gamma+1}\gamma}}\le C\,   | v|_{H^{\frac {2\gamma}{\gamma+1}}_{\frac {\gamma+1}\gamma}},
$$
we obtain
\DEQS
\lqq{ \EE\Big[ \sup_{0\le s\le t} K_2(s) \Big]}&&
\\
&\le&
\ep \EE\Big[ \int_0^ t |u(s)-u_0|_{L^{\gamma+1}}^{\gamma+1}\, ds\Big] +t\,\lk\{ \EE\Big[ \sup_{0\le s\le t} |\eta(s)|^{(\frac {\gamma+1}\gamma)^2}_{L^{\gamma+1}} \Big] \rk\}^\frac {\gamma+1}\gamma \lk\{
\EE\|\eta\|_{L^{\gamma+1}(0,T;L^{\gamma+1})}^{\gamma+1}\rk\}^\frac 1 \gamma.
\EEQS
Observe, that $u^\gamma-w^\gamma\le (u-w)^\gamma$.
Collecting all together we know that  for any $\ep_1,\ep_2>0$ there exist  constants $C_1,C_2>0$ such that
\DEQS
\lqq{\EE\Big[ \sup_{0\le s\le t}|u(s)-u_0|_{H^{-1}_2}^2\Big]+(\gamma+1)\EE \Big[ \int_0^ t|u(s)-u_0|^{\gamma+1}_{L^{\gamma+1}}\, ds\Big]}
&&
\\
&\le & \ep_1 \EE\Big[ \int_0^ t |u(s)-u_0|^{\gamma+1}_{L^{\gamma+1}}\, ds\Big]+C(\gamma,\ep_1)t^{\frac \gamma{\gamma+1}}
\, \lk\{ \EE |u_0|_{L^{\gamma+1}}^{\gamma+1} \rk\}
\\
&&{}+\ep_2 \EE\Big[ \int_0^ t |u(s)-u_0|_{L^{\gamma+1}}^{\gamma+1}\, ds\Big] +t\,\lk\{ \EE \Big[\sup_{0\le s\le t} |\eta(s)|^{(\frac {\gamma+1}\gamma)^2}_{L^{\gamma+1}}\Big] \rk\}^\frac {\gamma+1}\gamma \lk\{
\EE\|\eta\|_{L^{\gamma+1}(0,T;L^{\gamma+1})}^{\gamma+1}\rk\}^\frac 1 \gamma
\\
&\le & \ep_1 \EE \Big[ \int_0^ t |u(s)-u_0|^{\gamma+1}_{L^{\gamma+1}}\, ds\Big]
+C_1t^{\frac \gamma{\gamma+1}}
\, \lk\{ \EE |u_0|_{L^{\gamma+1}}^{\gamma+1} \rk\}
\\
&&{}+\ep_2 \int_0^ t \EE \Big[|u(s)-u_0|_{L^{\gamma+1}}^{\gamma+1}\, ds\Big] +C_2t\, R_1^\frac {\gamma+1}\gamma R_2^\frac 1 \gamma.
\EEQS
Taking $\ep_1$ and $\ep_2$ small enough 
gives the assertion. This finishes the proof of the proposition.

\end{proof}

\end{document}